%% file: P_Bermolen__M_Jonckheere__J_Sanders__Scaling_limits_and_Generic_Bounds_for_Exploration_Processes__v1__arXiv_version.tex
\documentclass[a4paper,11pt]{article}

\usepackage[english]{babel}

\usepackage{amsfonts,dsfont}
\usepackage{graphicx}
\usepackage{enumerate}
\usepackage{latexsym,amssymb,amsthm,amsxtra}
\usepackage{amsmath,mathrsfs}
\usepackage{tikz}
\usepackage{pgf}
\usepackage[acronym]{glossaries}
\usepackage{dcolumn}
\usepackage{stmaryrd}
\usepackage{pstricks}
\usepackage{pst-node}
\usepackage{multido}
\usepackage[normalem]{ulem}

\usepackage{tikz}
\usepackage{pgfplots}
\pgfplotsset{compat=newest}
\usepgfplotslibrary{groupplots}
\usepackage{subcaption}

\def\N{{\mathbb R}}
\def\G{{\mathcal{G}}}
\def\F{{\mathcal{F}}}
\def\({{\Bigl(}}
\def\){{\Bigr)}}

\usepackage{anysize}
\marginsize{3cm}{2.5cm}{2.5cm}{2.5cm} 

\newtheorem{theorem}{Theorem}[section]
\newtheorem{lemma}[theorem]{Lemma}
\newtheorem{prop}[theorem]{Proposition}

\newtheorem{coro}[theorem]{Corollary}

\newtheorem{assumption}[theorem]{Assumption}

\newcommand{\cardinality}[1]{ | #1 | }
\renewcommand{\d}[1]{\ensuremath{\operatorname{d}\!{#1}}}
\newcommand{\e}[1]{ {\mathrm{e}}^{ #1 } }

\newcommand{\expectation}[1]{ \mathbb{E} [ #1 ] }
\newcommand{\expectationBig}[1]{ \mathbb{E} \Bigl[ #1 \Bigr] }

\newcommand{\vect}[1]{ \boldsymbol{#1} }

\newcommand{\indicator}[1]{ \mathds{1} [ #1 ] }
\newcommand{\indicatorBig}[1]{ \mathds{1} \Bigl[ #1 \Bigr] }

\newcommand{\process}[2]{ \{ #1 \}_{ #2 } }
\newcommand{\smallO}[1]{ o(#1) }
\newcommand{\bigO}[1]{ O(#1) }

\newcommand{\probability}[1]{ \mathbb{P} [ #1 ] }
\newcommand{\probabilityBig}[1]{ \mathbb{P} \Bigl[ #1 \Bigr] }

\newcommand{\variance}[1]{ \mathrm{Var} [ #1 ] }

\def\eqcom#1{\overset{\textnormal{(#1)}}}

\newcommand{\refFigure}[1]{{\textrm{Figure~\ref{#1}}}}

\newcommand{\refEquation}[1]{{\textrm{\eqref{#1}}}}

\newcommand{\refCorollary}[1]{{\textrm{Corollary~\ref{#1}}}}
\newcommand{\refProposition}[1]{{\textrm{Proposition~\ref{#1}}}}
\newcommand{\refLemma}[1]{{\textrm{Lemma~\ref{#1}}}}

\newcommand{\refSection}[1]{{\textrm{Section~\ref{#1}}}}
\newcommand{\refAppendixSection}[1]{\textrm{Appendix~\ref{#1}}}

\newacronym{LLN}{LLN}{Law of Large Numbers}
\newacronym{CLT}{CLT}{Central Limit Theorem}
\newacronym{ER}{ER}{Erd\"{o}s--R\'{e}nyi}
\newacronym{RSA}{RSA}{random sequential adsorption}
\newacronym{RGG}{RGG}{Random Geometric Graph}

\makeglossaries

\title{Scaling limits and Generic Bounds for Exploration Processes}
\author{Paola Bermolen \and Matthieu Jonckheere \and Jaron Sanders}

\begin{document}

\maketitle


\begin{abstract}
We consider exploration algorithms of the random sequential adsorption type both for  homogeneous random graphs and random geometric graphs based on spatial Poisson processes. At each step, a vertex of the graph becomes active and its neighboring nodes become explored. Given an initial number of vertices $N$ growing to infinity, we study statistical properties of the proportion of explored nodes in time using scaling limits.
We obtain exact limits for homogeneous graphs and prove an explicit central limit theorem  for the final proportion of active nodes, known as the \emph{jamming constant}, through a diffusion approximation for the exploration process. 
We then focus on bounding the trajectories of such exploration processes on random geometric graphs, i.e.\ random sequential adsorption. As opposed to homogeneous random graphs, these do not allow for a reduction in dimensionality. Instead we build on a fundamental relationship between the number of explored nodes and the discovered volume in the spatial process, and obtain generic bounds: bounds that are independent of the dimension of space and the detailed shape of the volume associated to the discovered node. Lastly, we give two trajectorial interpretations of our bounds by constructing two coupled processes that have the same fluid limits.  
\end{abstract}

\section{Introduction}

Exploration algorithms occur in a wide a range of processes, for example in the evolution of parking processes \cite{bermolen_jc_2016}, as well as random sequential adsorption processes \cite{evans_random_1993}. A typical exploration algorithm works as follows: Assume there exists a binary relation between items $V =\{ 1, \ldots, N \}$, to which we associate a relation graph where nodes are items such that two items are neighbors if they are related. Let $A_t$ be the set of active items at step $t$ and $B_t$ the set of explored items.
We assume that, initially, $A_0=B_0=\{\emptyset\}$. Then we consider the following exploration process:
\begin{itemize}
\item[(i)] select $I_t \subseteq V \setminus \{ A_t \cup B_t \}$ and determine its neighbors in the set of non-explored items,
$\mathcal{N}_t \subseteq V \setminus \{ A_t \cup B_t \cup I_t \}$, and 
\item [(ii)] update $A_t$ and $B_t$ by setting $A_{t+1} = A_t \cup I_t$ and $B_{t+1}= B_t\cup \mathcal{N}_t$.
\end{itemize}
In case $I_t$ is a single item at each step, the algorithm discovers in a greedy manner independent sets of the relation graph. If $I_t$ is a subgraph, $\mathcal{N}_t$ is the set of non-active, non-explored neighbors of this subgraph at step $t$.

Exploration algorithms can also be identified in diverse applications. For example, they can be linked to communication procedures in communication networks \cite{evans_random_1993}. In this context, the relation graph might be the outcome of random spatial effects, i.e.\ nodes could be interacting through a geometry which itself can be random. This is certainly the case in wireless networks in which acceptable radio conditions have been defined based on the level of admissible interference between two competing nodes, so conditions that determine whether simultaneous communication is possible. These interference constraints can be modeled using a hardcore interference graph, in which an edge is present between two nodes if their radio conditions would impede a simultaneous communication \cite{bermolen_estimating_2016}. 

Relation graphs and associated exploration processes can also be used to describe blockade effects in complex systems of interacting particles. Consider for instance gases of ultracold atoms that can change between a so-called Rydberg state and a ground state \cite{gallagher_rydberg_1994}. The essential feature of these gases is that once an atom is in its Rydberg state, it prevents neighboring atoms from reaching their Rydberg state. By modeling the dynamics of the excitation process as an exploration process on a relation graph (specifically, the canonical \gls{ER} graph),  the statistical properties of the proportion of atoms ultimately in the Rydberg state were studied in \cite{sanders_sub-poissonian_2015}. This article formalizes and generalizes the heuristical arguments provided there.

The dimension needed to represent exploration algorithms as a Markov process is in general equal to the size of the underlying graph, $N$, and this impedes both the analysis and the ability to perform explicit calculations. To overcome these difficulties, we first assume that a homogeneity assumption holds for the relation between items. This allows us to describe the dynamics using a one-dimensional Markov process, instead of a prohibitively more complicated $N$-dimensional description. While this is a coarse simplification for many problems, it has the advantage that the analysis remains tractable. We then use classical tools of probability theory, specifically fluid limits and diffusion approximations, to derive computable characterizations of the performance of these exploration algorithms. In particular we provide not only convergence results but also error bounds. These techniques furthermore allow us to prove a \gls{LLN} as well as a \gls{CLT} for the proportion of active nodes in the jamming limit, i.e.\ at the end of the exploration.

While the techniques to obtain a \gls{LLN} for the exploration process of the \gls{ER} graph are well-known in the literature, see for instance the survey \cite{Darling2008}, our first contribution is to use such approaches to also obtain a \gls{LLN} and \gls{CLT} for the jamming constant (the final proportion of active nodes). Our results also cover exploration processes on other, sufficiently symmetric graphs. By next building on the recent results in \cite{berkes2014}, we can also give error bounds for the diffusion approximation, which in turn allows to construct confidence intervals for the jamming constant. These results have earlier been made available in the PhD thesis \cite{sanders_stochastic_2016}, as well as in the technical report \cite{bermolen_scaling_2015}.

Our new, second contribution consists of using the fluid and diffusion limits of exploration processes to obtain bounds for the more involved case of random geometric graphs. When the neighborhood relation is constructed using a spatial Poisson process and a geometric distance (say the Euclidean one), the correlations between edges of the graph have a much more complex structure. As underlined before, the exploration process can then not be reduced to the analysis of a one-dimensional Markov process. However, by building on a fundamental relationship between the number of explored nodes and the discovered volume in the spatial process (translating the properties of the Poisson process), we obtain bounds for the fluid limit of the exploration process. For both our lower and upper bound on the fluid limit, we propose a coupling of the original process with processes which have a lower and higher exploration rate, respectively.

The rest of the paper is structured as follows. In \refSection{sec:Random_adsorption_under_a_homogeneous_relation}, we set the model for homogeneous and symmetric random graphs and obtain classical scaling limits in this context. Then, building on diffusion approximations errors bounds, we provide diffusion approximations for both the exploration process and the jamming constant. In \refSection{sec:RGG}, we characterize an important asymptotic relation between the mean discovered volume and the mean number of explored points for random relation graphs associated to Poisson point processes. Building on this relation, we obtain asymptotic bounds for the exploration process and provide trajectorial couplings for these bounds. Finally, in \refSection{sec:Simulations}, we provide simulations of the exploration process and the jamming constant of the geometric random graph, which we compare to our bounds.

\section{Fluid limit and diffusion approximation for homogeneous graphs}
\label{sec:Random_adsorption_under_a_homogeneous_relation}

Assume from here on that precisely one item is selected in each step. Assume also the following homogeneity property on the graph $\G$:
\begin{assumption}
If $(\G_1, \G_2)$ is a partition of $\G$, then the mean number of edges from $\G_1$ to $\G_2$ depends only
on the size of the partition, i.e.\ $\cardinality{ \G_1 }$ and $\cardinality{ \G_2 }$.
\label{A1}
\end{assumption}

While Assumption (\ref{A1}) is not valid in cases such as random geometric graphs and random graphs with generic degree distribution, it is crucial to be able to describe the exploration process in one dimension. Assumption (\ref{A1}) can however be considered a reasonable approximation for many systems, and it is for example satisfied by \gls{ER} random graphs. We refer the reader to \cite{bermolen_jc_2016} for a study of scaling limits in infinite dimension that applies to a larger class of problems.

If we let $Z_n$ denote the number of explored items at step $n$, so $Z_n= \cardinality{ A_n \cup B_n }$, we have that
\begin{equation}
Z_0 = 0,
\quad \text{and}
\quad Z_n = Z_{n-1} + 1 + \xi_n.
\end{equation}
Here, $\xi_n$ denotes the number of neighbors that the selected item has at step $n$ in the remaining non-explored portion of the graph.

The distribution of $\xi_n$ depends under Assumption (\ref{A1}) only on $Z_{n-1}$, which we will denote by $\xi_{Z_{n-1}}$ with a slight abuse of notation. Assumption (\ref{A1}) also implies that the process $\process{ Z_n }{ n \in \N }$ is a discrete Markov process that takes values in $\{0,\dots, N\}$, is strictly increasing, and has in $N$ an absorbing state. The transition probabilities at step $n$ of the process $\process{ Z_n }{ n \in \N }$ are therefore given by
\begin{equation}
p_{xy}(n)
= \probability{ Z_n = y | Z_{n-1} = x }
= \probability{ \xi_x = y - x - 1 }
\quad
\text{for}
\quad
y > x.
\end{equation}

If we now denote by $p_N(\cdot,x)$ the distribution of the number of neighbors in $\G_2$ of any vertex $i\in \G_2$ given that $(\G_1,\G_2)$ is a partition of $\G$ with $|\G_1|=x$, the transition probabilities can be written as
\begin{equation}
p_{x, x+k+1}(n)
= \probability{ \xi_x=k }
= p_N(k,x),
\quad
\text{with}
\quad
k \geq 0.
\end{equation}
In case of the \gls{ER} random graph, in which an edge exists between a pair of nodes with probability $p$, the transition probabilities are given by the Binomial distribution, i.e.\
\begin{equation}
p_N(k,x) = \binom{N-x-1}{k} p^k (1-p)^{N-x-k-1}.
\end{equation}


Given a partition $(\G_1,\G_2)$ of $\G$ such that $|\G_1|=x$, we denote by $\gamma_N(x)$ the mean and by $\psi_N(x)$ the variance of the number of neighbors in $\G_2$ of a given vertex $i\in\G_2$, i.e.\
\begin{equation}
\gamma_N(x) = \sum_{k=0}^{N-1} k p_N(k,x),
\quad
\psi_N(x) = \sum_{k=0}^{N-1} (k-\gamma_N(x))^2 p_N(k,x).
\end{equation}
We also define $\bar{\gamma}_N = \sup_{x} \gamma_N(x)$, $\bar{\psi}_N = \sup_{x} \psi_N(x)$. 

Now consider the scaled process defined as the piece-wise constant trajectory process of
\begin{equation}
Z^{N}_t
= \frac{ Z_{[tN]} }{N}
\end{equation}
for all $t \geq 0$. Here, $[x]$ is the integer part of $x$. We derive a fluid limit for $Z_t^N$ in \refProposition{prop:LLN_comparing_Znt_and_Zt}. While the proof of convergence relies on classical techniques \cite{Darling2008}, we leverage these tools to also obtain error bounds along the way. The proof of \refProposition{prop:LLN_comparing_Znt_and_Zt} is deferred to \refAppendixSection{appendix:Proof_of_proposition_LLN_comparing_Znt_and_Zt}.

\begin{prop}
\label{prop:LLN_comparing_Znt_and_Zt}
If there exists a ($C_L$)-Lipschitz continuous function $\gamma$ on $\mathbb R^+$ such that
\begin{equation}
\sup_{x} \Bigl| \gamma_N(x) - \gamma \Bigl( \frac{x}{N} \Bigr) \Bigr|
\leq \delta_N,
\label{eq:def_gamma}
\end{equation}
then for $p > 1$ and $T > 0$,
\begin{equation}
\lVert \sup_{s \in [0,T]} |Z^{N}_s- z(s)| \rVert_p
\leq \Big( \delta_N T + \frac{1 + \bar{\gamma}_N }{N} + \kappa_p \Bigl\lVert \frac{M_{[TN]}}{N} \Bigr\rVert_p \Big) \e{ C_L T },
\label{eqn:Bound_on_Lp_norm_of_supremum_distance_ZsN_and_zs}
\end{equation}
where $\kappa_p = p / (p-1)$, and $z(t)$ denotes the solution to the deterministic differential equation
\begin{equation}
\dot{z}(t)
= 1 + \gamma( z(t) ),
\quad
\textrm{with}
\quad
z(0) = 0, 
\quad
\textrm{for}
\quad
t \le T^*=\min \{s : z(s)=1\}.
\label{eq:ODE.disc}
\end{equation}
Here, $M_n = Z_n - \sum_{i=0}^n ( 1 + \gamma_N(Z_i) )$ denotes a global martingale.

For $p = 2$, the bound reduces to
\begin{equation}
\lVert \sup_{t \in [0,T]} |Z^{N}(t)- z(t)| \rVert_2
\leq \left( \delta_N T + \frac{1+\bar{\gamma}_N}{ N} + 2 \sqrt{\frac{ \bar{\psi}_N T}{ N }} \right) \exp{(C_L T)}
= \omega_N.
\end{equation}
\end{prop}

\begin{coro}
\label{coro.fluid.disc}
If the distribution of the number of neighbors is such that $\delta_N \to 0$ as $N \to \infty$, $\bar{\gamma}_N = \smallO{N}$, and $\bar{\psi}_N = \smallO{N}$, then the scaled process $Z^N_t$ converges to $z(t)$ in $L^1$ uniformly on compact time intervals.
\end{coro}


We now proceed and derive a diffusion approximation theorem for the scaled number of explored items $Z_t^N$. The convergence proof relies similarly on classical techniques, which we again leverage to determine error bounds. To that end, we apply results of \cite{kurtz78} which are based on results by Koml\'{o}s--Major--Tusn\'{a}dy \cite{komlos_approximation_1975,komlos_approximation_1976}. The results in \cite{kurtz78} allow one to construct a Brownian motion and either a Poisson process or random walk on the same probability space. Since we are concerned with discrete time, we need to consider the random walk case, see also \cite{berkes2014}. In order to obtain explicit error bounds, we impose stronger assumptions on the transitions probabilities than would be needed when only proving convergence. Our proof of \refProposition{prop:Diffusion_approximation_with_error_bound} can be found in \refAppendixSection{appendix:Proof_of_proposition_Diffusion_approximation_with_error_bound}.

\begin{prop}
\label{prop:Diffusion_approximation_with_error_bound}
If there exists a function $p$ on $(\mathbb N,\mathbb R_+)$, and a sequence $(\epsilon_k)_{k=0,1,\ldots}$ such that
\begin{align}
|p_N(k,[Nx])- p(k,x)| & \leq \frac{\epsilon_k}{N}, \nonumber \\
|p(x,x+k)- p(y,y+k)| & \leq M \epsilon_k |x-y|, \nonumber \\
\sum_{k} k^2 |p(x,x+k)^{1/2}- p(y,y+k)^{1/2}|^2 & \leq M  |x-y|^2
\end{align}
and
\begin{equation}
\sum_{k} k \epsilon_k^{1/2} < \infty,
\end{equation}
and if $\gamma$ is twice differentiable with bounded first and second derivatives, then the process
\begin{equation}
W^N_t = \sqrt{N} ( Z^N_t-z(t) )
\label{eqn:Definition_of_WNt}
\end{equation}
converges in distribution towards $W_t$, the unique solution
of the stochastic differential equation
\begin{gather}
\d{W(t)}
= \gamma'(z(t)) W(t) \d{t} + \sqrt{\beta'(t)} \d{B_1(t)}.
\end{gather}
Here, $B_1(t)$ denotes a standard Brownian motion, $\beta(t) =  \int_{0}^t  \psi(z(s)) \d{s}$, and $z(t)$ is the solution of \refEquation{eq:ODE.disc}. Furthermore,
\begin{equation}
\expectationBig{ \sup_{t \le T} | W^N_t - W_t | }
\leq C \frac{\log(N)}{\sqrt{N}}.
\end{equation}
\end{prop}


\subsection{{LLN and CLT for the hitting time}}
\label{sec:LLN_and_CLT_for_the_hitting_time}

The exploration algorithm finishes at time
\begin{equation}
T_N^* = \inf\{ \tau \in \mathbb{N}_+ | Z_\tau = N \} \leq N < \infty,
\end{equation}
which is a hitting time for the Markov process. Since the algorithm adds precisely one node at each step, we have that the final number of active items is exactly $T^*_N$, i.e.\ $A_{T_N^*} = T_N^*$. Because we wish to determine the statistical properties of $A_{T_N^*}$, we will seek not only a first-order approximation for $T^*_N$, but also prove a central limit theorem result as the initial number of items $N$ goes to infinity.

Since the exploration process converges to the fluid limit $z(t)$, we can anticipate that an appropriately scaled hitting time $T_N^*$ converges to $T^*$, the solution to $z(T^*) = 1$. This intuition is formalized in the \GLS{LLN} result for $T_N^*/N$ in \refProposition{prop:Bound_on_L2_norm_of_TNstar_minus_Tstar}. Its proof is in \refAppendixSection{appendix:Proof_of_proposition_Bound_on_L2_norm_of_TNstar_minus_Tstar}.

\begin{prop}
\label{prop:Bound_on_L2_norm_of_TNstar_minus_Tstar}
For all $\delta > 0$,
\begin{equation}
\probabilityBig{ \Bigl| \frac{T^*_N}{N} - T^* \Bigr| \ge \delta }
\le \frac{2\omega_N}{\delta}.
\end{equation}
Moreover if $\gamma$ is continuous, non-increasing with $\gamma(1) = 0$, then there exists a constant $C$ for sufficiently small $\delta$ so that
\begin{equation}
\Big\lVert \frac{T^*_N}{N} - T^* \Big\rVert_2
\leq C \omega_N := \Omega_N
\end{equation}
\end{prop}

\begin{coro}
If the distribution of the number of neighbors is such that $\delta_N \to 0$ as $N \to \infty$, $\bar{\gamma}_N = \smallO{N}$, and $\bar{\psi}_N = \smallO{N}$, then the proportion of active items $T_N^* / N$ converges in $L^1$ to $T^*$.
\end{coro}

The \gls{LLN} in \refProposition{prop:Bound_on_L2_norm_of_TNstar_minus_Tstar} provides us formally with a candidate, $T^*$, around which to center $T_N^* / N$ and subsequently prove the \gls{CLT} result in \refProposition{prop:CLT_of_Tstar}. We defer to \refAppendixSection{appendix:Proof_of_proposition_CLT_of_Tstar} for its proof.

\begin{prop}
\label{prop:CLT_of_Tstar}
There exist constants $C_1, C_2$ such that the expectation of the random variable $\sqrt{N} ( T_N^* / N - T^* )$ centered around $- W_T^*$ is bounded by
\begin{align}
&
\expectationBig{ \Bigl| \sqrt{N} \Bigl( \frac{T_N^*}{N} - T^* \Bigr) + W_{T^*} \Bigr| }
\\ &
\leq C_1 \omega_N^2 \sqrt{N} + \Bigl( \bar{\psi}_N \Omega_N + \frac{\bar{\psi}_N}{N} \Bigr)^{\frac{1}{2}} + C_2 \frac{\log(N)}{\sqrt{N}} + \frac{1 + \bar{\gamma}_N}{\sqrt{N}}.
\nonumber
\end{align}
Moreover, if the distribution of the number of neighbors is such that $\delta_N = \smallO{ 1 / \sqrt{N} }$, $\bar{\gamma}_N = \smallO{ \sqrt{N} }$, and $\bar{\psi}_N = \smallO{ N^{1/4} }$, then
 $\sqrt{N} ( T_N^* / N - T^* )$ converges in $L^1$ to $W_{T^*}$ that is a centered Gaussian random variable with variance
\begin{equation}
\sigma^2
= m(T^*)
= \expectation{ W_{T^*}^2 },
\end{equation}
and where $m(t) = \expectation{ W_t^2 }$ solves the differential equation
\begin{equation}
\dot{m}(t)
= -2 \dot{\gamma}( z(t) ) m(t) + \dot{\beta}(t),
\quad
\textrm{with}
\quad
m(0) = 0.
\label{eq:ode_var_er_disc}
\end{equation}
\end{prop}

\subsection{Case: \glsentrytext{ER} random graph}

We now illustrate our results through an application of Propositions~\ref{prop:LLN_comparing_Znt_and_Zt}, \ref{prop:CLT_of_Tstar} to the case of the \gls{ER} random graph. We point interested readers to \cite{dhara_solvable_2016} for the analysis of another, more involved example of an exploration on a graph with sufficient symmetry.

Suppose that given $N$ the graph $\G = \G(N, c/N)$ is a sparse \gls{ER} graph, i.e.\ $p_N(\cdot,x)$ is the probability mass function of the binomial distribution $\textrm{Bin}(N-x-1, c/N)$ with $c > 0$.
Additionally, suppose that $N-1$ is Poisson distributed with parameter $h$. The mean and variance of $p_N(.,x)$ are then given by
\begin{equation}
\gamma_N(x)
= (N-x-1) \frac{c}{N},
\quad
\psi_N(x)
= (N-x-1) \frac{c}{N} \Bigl( 1- \frac{c}{N} \Bigr).
\end{equation}

Let $\gamma(x) = c(1-x)$. Condition \refEquation{eq:def_gamma} is then satisfied with $\delta_N = c / N$, and as Lipschitz constant $C_L=c$ suffices. Moreover, $\bar{\gamma}_N, \bar{\psi}_N \le c$. The deterministic differential equation in \refEquation{eq:ODE.disc} reads
\begin{equation}
\dot{z}(t) 
= 1 + c( 1-z(t) )
= (1+c) - c z(t), 
\quad
\textrm{with}
\quad z(0)=0.
\label{eq:ODE.er}
\end{equation}
This differential system can be explicitly solved, giving
\begin{equation}
z(t)
= \frac{1+c}{c} \Bigl( 1 -  \e{- c t}\Bigr).
\end{equation}
Observe that $\underset{t\to\infty}{\lim} z(t)= (1+c)/c > 1$, implying the existence of a finite, constant solution $T^*$ to $z(T^*) = 1$. Solving this equation, we find that
\begin{equation}
T^* 
= \frac{ \ln{(1+c)} }{c},
\end{equation}
which agrees with literature \cite{mcdiarmid}.

We next calculate $T^*$'s variance using \refProposition{prop:CLT_of_Tstar}. First, we verify its assumptions. Relations between the binomial coefficients and the Poisson distribution are well studied.
Defining
\begin{equation}
p_N(k,[xN])
= \binom{N-[Nx]-1}{k} \Bigl( \frac{c}{N} \Bigr)^k \Bigl( 1-\frac {c}{N} \Bigr)^{N-[Nx]-1},
\end{equation}
and using (for instance) the Stein--Chen method \cite{teerapabolaan_bound_2013}, we have that
\begin{equation}
|p_N(k,[xN])-p(k,x)|
\leq \frac{c}{N} p(k,x),
\end{equation}
which shows that the assumptions of Proposition~\ref{prop:Diffusion_approximation_with_error_bound} are satisfied. Moreover, the differential equation for $\beta(t)$ is given by
\begin{equation}
\dot{\beta}(t)
= \psi(z(s))=(1+c) \e{-ct} - 1,
\end{equation}
and the solution to \refEquation{eq:ode_var_er_disc} is then
\begin{equation}
m(t)
= \e{ - 2 ct } ( 1 - \e{ c t } )( \e{ ct } - 2c - 1 ) \frac{1}{2c},
\end{equation}
ultimately leading to
\begin{align}
\sigma^2
= \frac{ m(T^*) }{ ( 1-\gamma(1) )^2 }
= \frac{c}{2(c+1)^2}.
\end{align}

%
%
%
%

\section{Random Geometric Graph}\label{sec:RGG}

In this section we consider the problem of \gls{RSA} \cite{evans_random_1993}, which can be described as an exploration process on a \gls{RGG}. Due to the strong spatial correlation between points in \gls{RSA}, a quantitative analysis is notoriously difficult. We propose instead to use the fluid limits discussed in \refSection{sec:Random_adsorption_under_a_homogeneous_relation} as a stepping stone to obtain trajectorial bounds on the actual exploration process. As a consequence, we obtain a lower and upper bound for the jamming constant.

\subsection{\Gls{RSA} as an exploration process $\process{Z_n}{n \geq 0}$ on a \glsentrytext{RGG}}

Let $\Phi$ be a homogeneous Poisson point process with intensity $\lambda>0$ in a finite box $\mathcal{C}\subset \mathbb{R}^d$. We consider a \gls{RGG} $\mathcal{G}$ in which we identify each vertex with a point of $\Phi$, and there is an edge between every two points $X_i, X_j \in \Phi$ within distance $r>0$ of each other, i.e.\ $| X_i-X_j|\leq r$. The points $X_i$ and $X_j$ are then considered neighbors. Recall that given the number of points $N =\Phi(\mathcal{C})$ in $\mathcal{C}$, the points $X_1, X_2, \dots, X_N$ are independently and uniformly at random distributed in $\mathcal{C}$.

Let $Z_n$ denote the number of points (vertices of $\G$) explored at step $n$, and set $Z_0 = 0$. Consider any sequence of sets $\{ S_1, S_2, \ldots \}$ defined on $\mathcal{C}$. Let $E_n = \cup_{k\leq n} S_k$ denote the union of these sets until step $n$. Each $k$-th set $S_k$ corresponds to the new area of $\mathcal{C}$ being explored at step $k$, while $E_n$ corresponds to the total explored area up to and including step $n$.

In particular, for \gls{RSA} at each $n$-th step of the process, we choose a point $X_{i_n}$ uniformly at random among all non-explored points, and next consider $X_{i_n}$ and all of its neighbors in $S_n = B^*(X_{i_n}, r) \cap E_{n-1}^c$ explored. Here, $B^*(x, r) = B(x,r) \setminus \{ x \}$ denotes a ball of radius $r$ centered around $x$ but excluding its center $x$. The number of explored vertices up to time $n$,  $Z_n$, satisfies the stochastic recursion:
\begin{equation}
Z_{n} = 1 + Z_{n-1} + \Phi(S_n).
\label{eq:def_Z}
\end{equation}
Note that $Z_n$ thus always increases at least by one. The exploration process ends at the first time $T^*_N$ such that $Z_{T^*_N} = N$, i.e.\ when all points of $\G$ have been explored.

The process $Z_n$ is a discrete Markov process with respect to the filtration
\begin{equation}
\mathcal{F}_n 
= \sigma( \cup_{i \leq n} \{ Z_i, S_i \} ),
\end{equation}
and lives on a finite state space $\{ 0, 1, \dots, N\}$ with absorbing state $N$.
However the computation of its scaling limit and of the stopping time $T^*_N$ are prohibitively complicated,
because the need to track the history of the explored volume impedes a direct analysis of its drift.
We therefore instead resort to deriving upper and lower bounds for $Z_n$ and $T^*_N$.

\subsection{Fluid limit properties of $\process{Z_n}{n \geq 0}$}

In the following proposition, we state the fluid limit convergence result for the exploration process as proven by Penrose et al. \cite{penrose2002}. From this fundamental result, we derive a similar convergence result for the fraction of explored volume. Let $Z^N_t := Z_{[tN]} / N$ denote the scaled version of the exploration process $Z_n$ in \refEquation{eq:def_Z}.

\begin{prop}
\label{prop:Penroses_fluid_limit_on_RSA_and_our_volume_fluid_limit}
There exists a deterministic function $z : \mathbb R^+ \to [0,1]$ such that
\begin{equation}
Z^N_t \underset{N\to \infty}{\to} z(t)
\text{ almost surely, and in }
\cal L^1.
\label{eqn:Fluid_limit_for_ZNt_by_Penrose}
\end{equation}
As a consequence there also exists a function $\eta : \mathbb R^+ \to [0,1]$ such that
\begin{equation}
\frac{1}{|\mathcal{C}|} \sum_{i =1}^{[tN]} \cardinality{S_i}\underset{N\to \infty}{\to} \eta(t)
\text{ almost surely, and in }
\cal L^1.
\label{eqn:Fraction_of_explored_volume_fluid_limit}
\end{equation}
\end{prop}

\begin{proof}
The first result, \refEquation{eqn:Fluid_limit_for_ZNt_by_Penrose}, was proved in \cite{penrose_random_2001,penrose_random_2003}. The second claim, \refEquation{eqn:Fraction_of_explored_volume_fluid_limit}, follows using a similar proof methodology. Specifically, we aim to apply  \cite[Thm.~3.2]{penrose_random_2003}, which requires us to verify a few properties of the functional $t \to ( 1 / \cardinality{ \mathcal{C} } ) \sum_{i =1}^{[tN]} \cardinality{S_i}$.

First, observe that $t \to ( 1 / \cardinality{ \mathcal{C} } ) \sum_{i =1}^{[tN]} \cardinality{S_i}$ is translation invariant. Second, using percolation estimates, it is proved in \cite{penrose_random_2003} that that for all possible realizations of the marked point process there exists a (random) radius $R$ such that the exploration process at the origin (i.e.\ the state explored or not of a point placed at the origin) stays unmodified by any change in the realization of the marked point process outside a ball of radius $R$. This implies that the same property holds for the exploration of the volumes associated to the points of the point process. Lastly, note that $t \to ( 1 / \cardinality{ \mathcal{C} } ) \sum_{i =1}^{[tN]} \cardinality{S_i}$ is polynomially bounded since $( 1 / \cardinality{ \mathcal{C} } ) \sum_{i =1}^{[tN]} \cardinality{S_i} \le t \lambda\cardinality{ B(\cdot,r) }$.

Having established these three properties, we are now in position to apply \cite[Thm.~3.2]{penrose_random_2003}, which completes the proof.
\end{proof}

\subsection{Bounding \gls{RSA}'s fluid limit $z(t)$}
\label{sec:Bounding_RSAs_fluid_limit_zt}

Having obtained a fluid limit for the fraction of explored volume, we now prove a differential equation fundamentally relating the limiting fraction of explored volume $\eta(t)$ and the limiting fraction of explored points $z(t)$.

\begin{prop}
\label{prop:Fundamental_relation_between_the_fraction_of_explored_volume_and_the_limiting_fraction_of_explored_points}
For all $t \le T^*$,
\begin{equation}
\dot{z}(t) = 1 + (1-z(t))  \frac{\dot \eta(t)}{1-\eta(t)} 
= 1 + \dot{\eta}(t) \e{-\int_0^t \frac{1}{1-z(s)}ds}.
\label{prop:gamma}
\end{equation}

\end{prop}

\begin{proof}
First observe that $z(t)$ is a differentiable function. Indeed for all fixed $N$, $t$, and $h$,
\begin{equation}
\Bigl| \expectationBig{ \frac{ Z_{[(t+h)N]} }{N} } - \expectationBig{ \frac{ Z_{[tN]} }{N} } \Bigr|
\le v h,
\end{equation}
which implies that $z(t)$ is globally Lipschitz continuous. Hence $z(t)$ is almost everywhere differentiable with respect to the Lebesgue measure. Using the same argument, we can show that $\eta$ is differentiable.

Now recall that since $\Phi$ is a Poisson point process, given the number of points in $\mathcal{C}$, $N = \Phi(\mathcal{C})$,
the position of each point is independently and uniformly distributed in $\mathcal{C}$.
Thus by definition of the exploration dynamics, we obtain that
\begin{align}
\expectationBig{ Z^N_{[tN]} | \mathcal{F}_{[tN]-1} }
&
\eqcom{\ref{eqn:Fluid_limit_for_ZNt_by_Penrose}}= Z^N_{[tN]-1} + \frac{1}{N} + \frac{1}{N} \expectation{\Phi(S_{[tN]})  | \mathcal{F}_{[tN]-1} } 
\nonumber \\ &
= Z^N_{[tN]-1} + \frac{1}{N} + \frac{1}{N} \expectation{ \mathrm{Bin}( N - Z_{[tN]-1} - 1, p_i ) | \mathcal{F}_{[tN]-1} },
\label{eqn:Deconstruction_of_Zn_into_Binomials}
\end{align}
where $p_i = \cardinality{ S_i } / ( |\mathcal{C}| - \sum_{j \leq [tN]-1 } \cardinality{ S_j } )$. Since $|\mathcal{C}| = N / \lambda$, it follows that
\begin{align}
\expectationBig{ N  ( Z^N_{[tN]} - Z^N_{[tN]-1} )  \big| \mathcal{F}_{[tN]-1} }
&
= 1 +  \Bigl( 1 - Z^N_{[tN]-1} - \frac{1}{N} \Bigr) \frac{\lambda \expectation{ \cardinality{ S_{[tN]} } | \F_{[tN]-1} } }{ 1 - \frac{1}{N}\sum_{j \leq [tN]-1} \cardinality{ S_j } }.
\label{eqn:Appearance_of_prefluid_limit_in_the_deconstruction_of_Zn_into_Binomials}
\end{align}
Hence writing $\Gamma_N(i) = (1/N) \sum_{j \leq i} \cardinality{ S_j } $, we obtain after summing these differences that
\begin{align}
\expectationBig{ \frac{ Z_{[tN]} }{N}  \big| \mathcal{F}_{[tN]-1} } 
&
= \frac{[tN]}{N} + \frac{1}{N} \sum_{i=1}^{[tN]} \Bigl( 1 - Z^N_{i-1} - \frac{1}{N} \Bigr) \frac{\lambda N \expectation{ \Gamma_N(i)-\Gamma_N(i-1) | \F_{[tN]-1} } }{ 1 - \Gamma_N(i - 1)} ,\\
& = \frac{[tN]}{N} +  \sum_{i=1}^{[tN]} \Bigl( 1 - Z^N_{i-1} - \frac{1}{N} \Bigr) \frac{\lambda \expectation{ \Gamma_N(i)-\Gamma_N(i-1) | \F_{[tN]-1} } }{ 1 - \Gamma_N(i - 1)} ,
\label{eqn:Appearancev2}
\end{align}

Using the almost sure convergence of both $Z^N_t$ and $\Gamma_N$ to continuous functions, and since both functional 
are increasing in time, we have uniform almost sure convergence on compact (macroscopic) time intervals.
As a consequence, we obtain the almost sure convergence of the right-hand side of the previous equality to the Riemann-Stieljes integral 
$\int_{0}^t ( 1-z_s ) / ( 1-\eta_s ) \d{\eta_s}$ which is equal to $ \int_{0}^t ( 1-z_s ) / ( 1-\eta_s ) \dot{\eta}_s \d{s}$,
using the differentiability of $\eta$. Therefore,
\begin{equation}
\dot z(t)
= 1 + (1-z(t))  \frac{\dot \eta(t)}{1-\eta(t)}
\quad
\textrm{or equivalently}
\quad
\dot z(t)
= 1 + \dot \eta(t) \e{ -\int_0^t \frac{1}{1-z(s)}ds }.
\label{eq:dotz}
\end{equation}
This concludes the proof.
\end{proof}

It is important to note that we do not have an explicit expression available for either $z(t)$ or $\eta(t)$, hence the differential system is not solvable. But now that we have with \refProposition{prop:Fundamental_relation_between_the_fraction_of_explored_volume_and_the_limiting_fraction_of_explored_points} a relation between the limiting fraction of explored volume $\eta(t)$ and the limiting fraction of explored vertices $z(t)$ at our disposal, we can find upper and lower bounds $u(t)$, $l(t)$ for $z(t)$ by deriving upper and lower bounds for $\eta(t)$. These upper and lower bounds immediately also provide us with lower and upper bound for the jamming constant, since
\begin{equation}
T^{\mathrm{lower}} := \inf \{ t > 0 | u(t) = 1 \}
\leq T^*
\leq \inf \{ t > 0 | l(t) = 1 \} =: T^{\mathrm{upper}}.
\end{equation}

\begin{prop}
\label{prop:Lower_and_upper_bounds_for_the_fluid_limit_of_the_spatial_process}
For $t\in[0,T^*]$,
\begin{equation}
l(t) \leq z(t) \leq u(t),
\end{equation}
where $l(t)$ and $u(t)$ are the solutions to
\begin{equation}
\dot l(t)
= 1+ c \Bigl( 1 - \frac{3ct}{ 1 - l(t) } \e{ - \int_0^t \frac{1}{1-l(s)} \d{s}} \Bigr) \e{ -\int_0^t \frac{1}{1-l(s)} \d{s} },
\quad
l(0) = 0,
\label{eq:wt}
\end{equation}
and
\begin{equation}
\dot u(t)
= 1+ c \e{-\int_0^t \frac{1}{1-u(s)} \d{s} },
\quad
u(0) = 0,
\label{eq:yt}
\end{equation}
i.e.\ $u(t) = c t + ( t - (1/c) ) \ln{ ( 1 - c t ) }$, respectively. Here, $c = \lambda \cardinality{ B(\cdot,r) }$.
\end{prop}

\begin{coro}
\label{cor:Limiting_behavior_of_the_spatial_fluid_limit_as_c_to_zero}
As $c \downarrow 0$, the spatial process has fluid limit $z(t) = (1+c) t - \tfrac{1}{2} c t^2 + \bigO{c^2}$.
\end{coro}

The proof of \refCorollary{cor:Limiting_behavior_of_the_spatial_fluid_limit_as_c_to_zero} can be found in \refAppendixSection{appendix:Proof_of_corollary_Limiting_behavior_of_the_spatial_fluid_limit_as_c_to_zero}. We now prove \refProposition{prop:Lower_and_upper_bounds_for_the_fluid_limit_of_the_spatial_process}, which is actually a direct consequence of the volume bounds in \refLemma{lemma:volume_bounds}.

\begin{lemma}
\label{lemma:volume_bounds}
For every step $k = 1, \ldots, N$, with $v = \cardinality{ B(\cdot,r) }$,
\begin{equation}
v \Bigl( 1- \frac{(k-1)3v}{|\mathcal{C}|-\sum_{j<k} S_j } \Bigr)
\leq \expectation{ \cardinality{ S_k } }
\leq v.
\label{eqn:volume_bounds}
\end{equation}
\end{lemma}

\begin{proof}
The right inequality in \refEquation{eqn:volume_bounds} is immediate by noting that $\cardinality{ S_i^N } = \cardinality{ B^*(X_{i},r) \cap E_{i-1}^c } \leq \cardinality{ B^*(\cdot,r) }$. The fluid limit upper bound in  \refProposition{prop:Lower_and_upper_bounds_for_the_fluid_limit_of_the_spatial_process} then follows from $\eta$'s definition, since
\begin{equation}
\dot \eta(t)
= \lim_{N\to\infty} \lambda \expectation{ \cardinality{ S_{[tN]} } }
\leq \lambda v = c.
\label{eqn:Upper_bound_on_dot_gamma}
\end{equation}

We next prove the left inequality in \refEquation{eqn:volume_bounds}. Let $X_{i_k}$ be the selected point at step $k$ and define the event
\begin{equation*}
\mathcal{E}_{k} = \{ B^*(X_{i_k},r) \cap E_{k-1} = \emptyset \}.
\end{equation*}
Consequently, after decomposing and by strict positivity of $\cardinality{ S_k }$,
\begin{equation}
\expectation{ \cardinality{ S_k } }
= \expectation{ \cardinality{ S_k } | \mathcal{E}_{k} } \probability{ \mathcal{E}_{k} } + \expectation{ \cardinality{ S_k } | \mathcal{E}_{k}^c } \probability{ \mathcal{E}_{k}^c }
\geq \expectation{ \cardinality{ S_k } | \mathcal{E}_{k} } \probability{ \mathcal{E}_{k} }
= v \probability{ \mathcal{E}_{k} }.
\label{eqn:Decomposition_of_perimeter_event}
\end{equation}
We next use that the unexplored points are uniformly distributed in $E_{k-1}^c$ to obtain that
\begin{align}
\probability{ \mathcal{E}_{k} }
\geq 1 - (k-1) \probability{ \exists_{j < k} : X_{i_k} \in B(X_{i_j}, 2r)\setminus B(X_{i_j}, r) }
= 1 - \frac{ (k-1)3v }{ \cardinality{C} - \sum_{j<k} \cardinality{ S_j } }.
\label{eqn:Lower_bound_on_the_probability_of_event_Ek0}
\end{align}
Together with \refEquation{eqn:Decomposition_of_perimeter_event}, this gives the left inequality in \refEquation{eqn:volume_bounds}. To prove the fluid limit lower bound in \refProposition{prop:Lower_and_upper_bounds_for_the_fluid_limit_of_the_spatial_process}, gather that
\begin{align}
\dot \eta(t)
&
= \lim_{N\to\infty} \lambda \expectation{ \cardinality{ S_{[tN]} } }
\nonumber \\ &
\eqcom{\ref{eqn:Lower_bound_on_the_probability_of_event_Ek0}}\geq \lambda v \underset{N\to\infty}{\lim} \probability{ \mathcal{E}_{[tN]} }
= c \Bigl(  1- \frac{ 3ct }{ 1 - \eta(t) } \Bigr)  \eqcom{\ref{eq:dotz}}= c\Bigl( 1- \frac{3ct}{1-z(t)} \Bigr)e^{-\int_0^t \frac{1}{1-z(s)}ds}
\label{eq:limit_pAn}
\end{align}
This completes the proof.
\end{proof}

\subsection{Trajectorial bounds for $\process{Z_n}{n \geq 0}$}

We now construct trajectorial upper and lower bounds for the process $\process{Z_n}{n \geq 0}$ by constructing two couplings that have the same fluid limits as the upper and lower bound in \refSection{sec:Bounding_RSAs_fluid_limit_zt}. These lead in turn to lower and upper bounds for the jamming constant $T_Z^* = \inf\{ t \geq 0 | Z_t = N \}$, respectively.

\subsubsection{Upper bound process $\process{ U_n }{n \geq 0}$}

We define a new process $\process{ U_n }{n \geq 0}$ that corresponds to an exploration process of $\G$ with a higher discovery rate. As before, at each step (say step $n$) we choose a point $X_{i_n}$ uniformly at random among the non-explored points of the process $\process{ Z_n }{n \geq 0}$. There are now two possible situations:
\begin{description}
\item[Case 1: $X_{i_n}$ is also unexplored for $\process{U_n}{}$.] We define a new set $\tilde S_n= B^*(X_{i_n}, \tilde r_n)\cap \tilde E^c_{n-1} $, with $\tilde r_n$ chosen such that the area of $\tilde S_n$ and that of a free ball $v = \cardinality{ B^*(\cdot,r) }$ coincide, that is $| \tilde S_n| = v$. Here $\tilde E^c_{i} = \cup_{j=1}^{i-1} \tilde S_j$.
\item[Case 2: $X_{i_n}$ is already explored for $\process{U_n}{}$.] We now choose a different point $X'_{i_n}$ uniformly at random from the non-explored points of the process $\process{U_n}{}$, and consider it instead as the point that will be explored for the process $\process{U_n}{}$. We then proceed as in the previous case, by again letting $\tilde S_n = B^*(X_{i_n}, \tilde r_n) \cap \tilde E^c_{n-1}$ with $\tilde{r}_n$ such that $\cardinality{ \tilde{S}_n } = v$.
\end{description}
In both cases, the exploration process is updated -- as for the original process -- according  to  the following recursion equation:
\begin{equation}
U_{n} = 1 + U_{n-1} + \Phi(\tilde S_n).
\end{equation}
Note that crucially both $S_n$ and $\tilde S_n$ are constructed using the point $X_{i_n}$ initially chosen by the process $\process{Z_n}{n \geq 0}$. 

Now let $T_U^* = \inf\{ t > 0 | U_t = N \}$ be the time at which the coupled process $\process{U_n}{n \geq 0}$ completes the exploration of the graph. The coupling has been constructed such that the following holds.

\begin{prop}
The process $\process{ U_n }{ n \geq 0}$ is such that $U_n \geq Z_n$ almost surely, and $T_U^* \leq T_Z^*$ with probability one.
\label{prop:general_bound}
\end{prop}

\begin{proof}
We couple the two explorations processes by using the same spatial point process $\Phi$ and the same exploration order for both processes. Next, we prove that the two processes stay ordered using induction: for the first step $n=1$, both process coincide and inequality holds. Consider next any $n > 1$ for which $U_{n-1} \geq Z_{n-1}$. By construction of $S_n$ and $\tilde S_n$ we obtain that $U_n = 1+ U_{n-1} + \Phi(\tilde S_n) \geq 1+ Z_{n-1} + \Phi(S_n) = Z_n$ almost surely.

Since we have now shown that at each step $U_n \geq Z_n$ almost surely, it follows immediately that $T_U^* \leq T_Z^*$ with probability one. This completes the proof.
\end{proof}

While the transition probabilities of the process $\process{ Z_n }{ n \geq 1 }$ cannot readily be explicitly calculated due to the spatial correlations between the points $\{ X_1, X_2, \ldots \}$, they can be calculated for the process $\process{ U_n }{ n \geq 1 }$ , as we prove in the \refProposition{prop:Fluid_limit_of_upper_bound_coupled_process_Un}.

\begin{prop}
\label{prop:Fluid_limit_of_upper_bound_coupled_process_Un}
Consider the scaled process $U_t^N = U_{[tN]} / N$, then for $T>0$,
\begin{equation}
\lim_{N \to \infty} \expectation{ \sup_{s\in[0,T]} | U_s^N - u(s) | }
= 0,
\end{equation}
where $u(t)$ is the solution to \refEquation{eq:yt}.
\end{prop}

\begin{proof}
By \refProposition{prop:general_bound} we have that $Z_n\leq U_n$ almost surely. This implies that the corresponding fluid limits verify the same inequality, i.e.\ $z(t)\leq u(t)$. The convergence of the bounding process towards its fluid limit follows from the arguments given in \refSection{sec:Random_adsorption_under_a_homogeneous_relation}, see in particular the arguments surrounding \refEquation{eqn:Bound_on_Lp_norm_of_supremum_distance_ZsN_and_zs}.
\end{proof}

\subsubsection{Lower bound process $\process{ L_n }{ n \geq 0}$}

Simultaneously with the construction of the process $\process{ Z_n }{ n \geq 0 }$, we are now going to construct a new process $\process{L_n}{n \geq 0}$ that also corresponds to an exploration on $\G$ but then with a lower discovery rate. This then provides us with a trajectorial lower bound on $Z_n$, leading to an upper bound for the jamming constant.

Before choosing a new point $X_{i_n}$ to explore for step $n$, consider the event
\begin{equation}
\mathcal{P}(d_n)
= \Bigl\{ X_{i_n} \in \mathcal{C} \backslash \bigcup_{j < n} B(X_{i,j}, r + d_n) \Bigr\},
\end{equation}
where $d_n \geq r$ denotes a perimeter radius. This event has two properties, namely that:
\begin{itemize}
\item[(i)] $\probability{ \mathcal{P}(d_n) }$ is decreasing and continuous in $d_n$ and, 
\item[(ii)] $\probability{ \mathcal{P}(r) } = 1 - P(X_{i_n} \in \cup_{j<n} B(X_{i_j},2r)) \geq \max\left\{ 0, 1 - \frac{ 3(n-1)v }{   \cardinality{C} - \sum_{j<n} S_j } \right\} =: \alpha_n$. 
\end{itemize}
Hence there exists a $d_n \geq r$ such that $\probability{ \mathcal{P}(d_n) } = \alpha_n$. We choose $d_n$ as such. Now choose a point $X_{i_n}$ uniformly at random among the non-explored points \emph{according to} the process $\process{ Z_n }{ n \geq 0 }$, i.e.\
\begin{equation}
X_{i_n}
\overset{d}= \mathrm{Unif}\Bigl( \mathcal{C} \backslash \bigcup_{j < n} B( X_{i_j}, r ) \Bigr).
\end{equation}
While for the process $\process{ Z_n }{ n \geq 0 }$ the point $X_{i_n}$ and its neighbors are always added to the set of explored points, for the process $\process{L_n}{n \geq 0}$ we add $X_{i_n}$'s neighbors to the set of explored points only if $\mathcal{P}(d_n)$ turned out to be true. I.e.\
\begin{description}
\item[Case 1: $\mathcal{P}(d_n)$ is true.] Both $X_{i_n}$ as well as its neighbors will be considered explored. Since $d_n \geq r$, it must be that $\tilde{S}_n = B^*( X_{i_n}, r )$.
\item[Case 2: $\mathcal{P}(d_n)$ is false.] Only the point $X_{i_n}$ will be considered explored, i.e.\ $\tilde{S}_n = \phi$.
\end{description}
Consequentially, the exploration process $\process{L_n}{n \geq 0}$ is updated similarly to the process $\process{ Z_n }{ n \geq 0 }$, i.e.\ $L_n = 1 + L_{n-1} + \Phi(\tilde S_n)$, but note that now
\begin{equation}
\cardinality{ \tilde{S}_n }
=
\begin{cases}
0 & \text{ with probability } 1- \alpha_n, \\
v & \text{ with probability } \alpha_n.
\end{cases}
\end{equation}

A proof similar to that of \refProposition{prop:general_bound} implies that $L_n \leq Z_n$ almost surely. We now again consider its scaled version $L_t^N = L_{[tN]} / N$ and analyze its limit behavior as $N \to \infty$.

\begin{prop}
For $T>0$,
\begin{equation}
\lim_{N \to \infty}
\expectation{ \sup_{s\in[0,T]} | L_s^N - l(s) | }
= 0,
\end{equation}
where $l(t)$ is the solution to \refEquation{eq:wt}.
\end{prop}

\begin{proof}
Using similar arguments as in the proof of \refProposition{prop:gamma}, we can prove that
\begin{equation}
\dot l(t)
= 1 + {\dot \eta_l(t)} \e{ -\int_0^t \frac{1}{1-l(s)} \d{s} }
\end{equation}
where
\begin{equation}
\dot \eta_l(t)
= \lim_{N \to \infty} \expectation{ \lambda \tilde v^N_{[tN]} }
\quad
\text{with}
\quad \tilde v_i^N = \expectation{ \cardinality{ \tilde S_i } |\mathcal{F}_{i-1} }
= v \alpha_i.
\end{equation}
Then by definition of the exploration process $\process{L_n}{n \geq 0}$ and equation \refEquation{eq:limit_pAn},
\begin{equation}
\dot \eta_l(t)
= \lambda v \lim_{N \to \infty} \probability{ \alpha_{[tN]} }
= c \Bigl(  1- \frac{3ct}{ 1-l(t) } \e{- \int_0^t \frac{1}{1-l(s)} \d{s}} \Bigr).
\end{equation}
This concludes the proof.
\end{proof}

\section{Simulation Results}
\label{sec:Simulations}

We now verify our results by simulating the exploration process $\process{Z_n}{n \geq 0}$ of \refSection{sec:Bounding_RSAs_fluid_limit_zt}. We simulate the two-dimensional variant in a box of size $l \times w$  with fixed density $\lambda > 0$, and then choosing for each of the $N = \lambda l w$ particles a position distributed uniformly at random. As part of the initialization, we calculate the distance matrix $D_{i,j} := | \vect{r}_i - \vect{r}_j |$ for all particles $i,j \in \{ 1, \ldots, n \}$, i.e.\ there are no periodic boundary conditions (the effects of which are negligible for large $N$), and set $\mathcal{A}_0 = \phi$. We then run the exploration process: At time $n+1$, we select a non-explored particle $v \in \mathcal{A}_n^c$, identify all of its unexplored neighbors $\mathcal{N}_v = \{ w \in \mathcal{A}_n^c | D_{v,w} < r \}$, and update $\mathcal{A}_{n+1} = \mathcal{A}_n \cup \{ v \} \cup \mathcal{N}_v$. This is repeated until all particles have been explored at time $\tau$ and consequently $\mathcal{A}_\tau^c = \phi$.

On the left in \refFigure{fig:Two_dimensional_jamming_state_and_its_scaled_process_with_bounds} we depict one resulting jammed state at time $t = T_Z^*$. On the right, we show the corresponding scaled process $Z^N_t$ together with our lower and upper bounds. Note that our bounds are tighter for small $t \lesssim 0.2$, and become looser for $t \gtrsim 0.2$. This happens because initially, newly arriving disks do not overlap with already deposited disks. Notice also that our lower bound becomes linear from $t \approx 0.2$ onward. This corresponds to us deleting only one particle at a time in the exploration process, and occurs because this is where our lower bound of \refLemma{lemma:volume_bounds} becomes invalid. We expect that if we can find a tighter lower bound than \refLemma{lemma:volume_bounds}, for example by using higher-order statistical information on the overlapping area between disks, this point should move to the right so that our lower bound for the scaled process will improve.

\begin{figure}[!hbtp]
\centering
\footnotesize
\begin{subfigure}[t]{0.33\textwidth}
\centering
\input{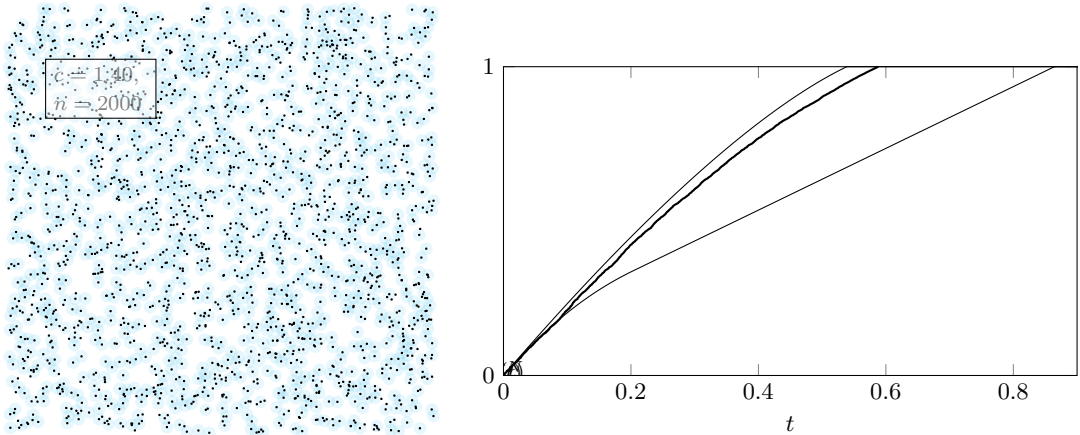}
\end{subfigure}
\begin{subfigure}[t]{0.66\textwidth}
\centering
\input{Figure__Scaled_process_and_lower_and_upper_bounds_of_fluid_limit}
\end{subfigure}
\caption{The resulting jammed-state for a two-dimensional simulation of the spatial process (left). The corresponding scaled process $Z^N_t$ as a function of time, together with our lower bound $l(t)$ and upper bound $u(t)$ (right). Here, $c \approx 1.40$ and $N = 2000$.}
\label{fig:Two_dimensional_jamming_state_and_its_scaled_process_with_bounds}
\end{figure}

In \refFigure{fig:Jamming_constant_as_function_of_c}, we have plotted our upper bound for the jamming constant $T^{\textrm{upper}}$, our lower bound $T^{\textrm{lower}}$, the Erd\"{o}s--R\'{e}nyi solution $( \ln{(1+c)} ) / c$, as well as an average of the jamming constant $T_Z^* / N$ obtained from $20$ samples per value of $c$, for different values of $c$. The upper and lower bounds for the jamming constant are obtained by numerically solving the differential equations in \refEquation{eq:wt} and \refEquation{eq:yt}, respectively\footnote{We solved \refEquation{eq:wt} numerically by reformulating it as a system of differential equations. Specifically, we solved $\dot{w_1}(t) = 1 + \max{ \{ 0, c ( 1 - ( 3ct w_2(t) ) / ( 1 - w_1(t) ) ) w_2(t) \} }$, and $\dot{w_2}(t) = - w_2(t) / ( 1 - w_1(t) )$ for $w_1(t)$ with initial conditions $w_1(0) = 0$, $w_2(0) = 1$.}. The dashed line indicates a $99\%$ confidence interval. Here, $N = 1000$.

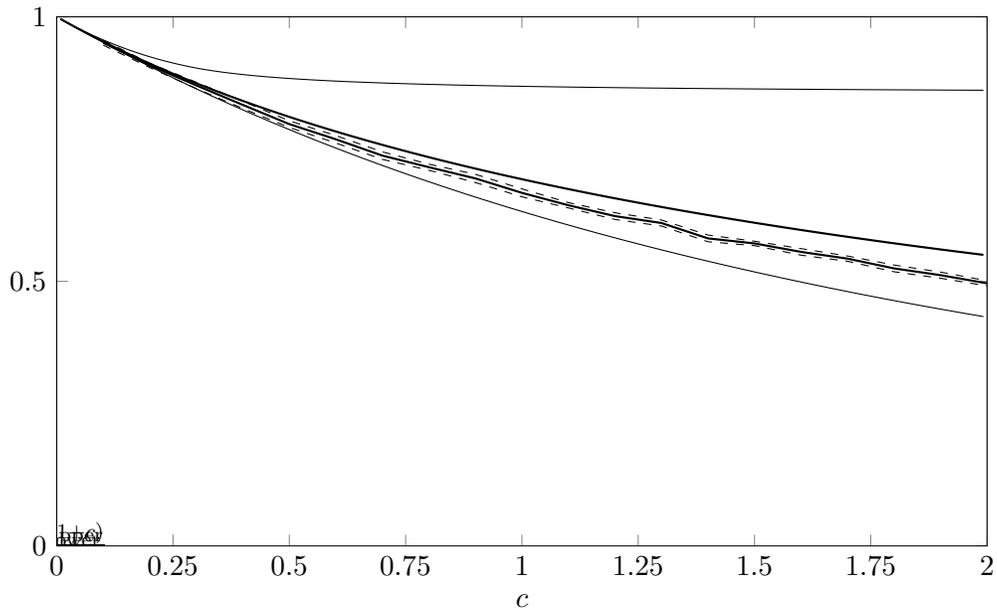
\begin{figure}[!hbtp]
\centering
\input{Figure__Jamming_constant_as_function_of_c}
\caption{Sample averages of the jamming constant as a function of $c$, together with our lower and upper bounds, and the Erd\"{o}s--R\'{e}nyi solution $( \ln{(1+c)} ) / c$.}
\label{fig:Jamming_constant_as_function_of_c}
\end{figure}

\refFigure{fig:Jamming_constant_as_function_of_c} illustrates, as \refCorollary{cor:Limiting_behavior_of_the_spatial_fluid_limit_as_c_to_zero} has proven, that our lower and upper bounds become tight as $c \downarrow 0$. As $c$ increases, our bounds loosen, because of the increased overlap between deposited disks. The Erd\"{o}s--R\'{e}nyi solution is surprisingly close to the actual jamming constant, particularly when we realize there is no geometry associated with the exploration process. The graph suggests that the Erd\"{o}s--R\'{e}nyi solution is an upper bound for the jamming constant, tighter than our upper bound. To prove this, it would be sufficient  (but not neccesary) to show that $\dot{\eta} / (1-\eta)$ is upper bounded by $c$. Our preliminary investigations have proven inconclusive thus far, but we believe the observation to be of interest for future research.

\section{Conclusions}
\label{sec:Conclusions}

We have shown that for sufficiently symmetric random graphs, exploration processes and their jamming limits can be asymptotically described with dynamical systems of lower complexity and with diffusion approximations. For geometric random graphs corresponding to the exploration of a marked Poisson process where this symmetry breaks, we build on fundamental properties of the Poisson process to provide generic bounds for the scaling limits, together with associated trajectorial coupling. These bounds are independent of the shape associated with the exploration and to the dimension of the point process. As future work, we want to investigate tighter bounds by taking into account both dimension and shapes, and we want to link these bounds to the scaling limits of symmetric random graphs.

\bibliographystyle{plain}
\bibliography{Bibliography}

\section*{Appendix}
\appendix

\section{Proof of \refProposition{prop:LLN_comparing_Znt_and_Zt}}
\label{appendix:Proof_of_proposition_LLN_comparing_Znt_and_Zt}

\begin{proof}
Doob's martingale decomposition \cite{steele_stochastic_2001} for the Markov process $\process{Z_n}{n \geq 0}$ gives that for $n \geq 0$,
\begin{equation} \label{eq:martingale_decomp_Z}
Z_n
= \sum_{i=0}^n ( 1 + \gamma_N(Z_i) ) + M_n.
\end{equation}
Here, we have used that $Z_0 = 0$, and $M_n$ denotes a local martingale that is actually a global martingale since the state space is finite.

We will now examine the scaled random variable $Z_t^N$, for which
\begin{align}
Z^N_t
&
= \frac{Z_{[tN]}}{N}
= \frac{1}{N} \sum_{i=0}^{[tN]} \bigl( 1+\gamma_N(Z_i) \bigr) + \frac{M_{[tN]}}{N}
\label{eq:martingale_decomp_Zn}
\\ &
\eqcom{i}= \nonumber \frac{1}{N} \int_{0}^{[tN]} \bigl( 1+\gamma_N(Z_s) \bigr) \d{s} + \frac{M_{[tN]}}{N}
\eqcom{ii}= \int_{0}^{ \frac{[tN]}{N} } \bigl( 1 + \gamma_N(Z_{uN}) \bigr) \d{u} + M^N_t,
\nonumber
\end{align}
since we (i) view each trajectory as being path-wise continuous, and (ii) use the change of variables $u = s / N$, and introduce the notation
$
M^N_t
= M_{[tN]} / N
$
for a scaled martingale.

We can replace the integral $\int_0^{[tN]/N} \cdots \d{u} $ by the integral $\int_0^t \cdots \d{u}$, which introduces an error $\Delta_{N,t}$. Specifically, we can write
\begin{equation}
\int_0^{ \frac{[tN]}{N} } \bigl( 1 + \gamma_N(Z_{uN}) \bigr) \d{u}
= \int_0^t \bigl( 1 + \gamma_N(Z_{uN}) \bigr) \d{u} + \Delta_{N,t}
\end{equation}
where
\begin{equation}
\Delta_{N,t}
= \int_0^{ \frac{[tN]}{N} } \bigl( 1 + \gamma_N(Z_{uN}) \bigr) \d{u} - \int_0^t \bigl( 1 + \gamma_N(Z_{uN}) \bigr) \d{u}.
\end{equation}
For large $N$ such replacement has negligible impact, since independently of $t$,
\begin{equation}
| \Delta_{N,t} |
\leq \sup_{u\in[0,1]} \{ 1 + \gamma_N(Z_{uN}) \} \Bigl| \frac{[tN]}{N} - t \Bigr|
\leq \frac{1 + \bar{\gamma}_N }{N},
\label{eqn:Bound_on_the_error_by_a_replacement_independent_of_t}
\end{equation}
where in the last inequality we have used that $\bar{\gamma}_N = \sup_x \gamma_N(x)$ and $| [tN] - tN | \leq 1$.

Using (i) the integral version of \refEquation{eq:ODE.disc}, the triangle inequality \cite{rudin_real_1987}, and (ii) Lipschitz continuity of $\gamma$, condition \refEquation{eq:def_gamma} and bound \refEquation{eqn:Bound_on_the_error_by_a_replacement_independent_of_t}, we find that
\begin{align}
\sup_{s\in[0,t]}| Z^N_s - z(s) |
&
\eqcom{i}\leq \sup_{s \in [0,t]} \Bigl( \int_0^{s} \bigl| \gamma_N(Z_{uN}) - \gamma(z(u)) \bigr| \d{u} + |\Delta_{N,s}| + | M_s^N | \Bigr)
\label{eqn:First_bound_on_sup_of_ZNs_and_Zs}
\\ &
\eqcom{ii}\leq C_L \int_0^{t} \sup_{u\in[0,s]}| Z^N_u - z(u)| \d{s} + \delta_N t + \frac{1+\bar{\gamma}_N}{N} + \sup_{s\in[0,t]}|M_s^N|.
\nonumber
\end{align}
Next, we define $\epsilon_N(T)= \sup_{s \in [0,T]} |Z^{N}_s- z(s)|$ for notational convenience and to prepare for an application of Gr\"{o}nwall's lemma \cite{steele_stochastic_2001}. Eq.\ \refEquation{eqn:First_bound_on_sup_of_ZNs_and_Zs} then shortens for $T > 0$ to
\begin{equation}
\epsilon_N(T) \leq \delta_N T + \frac{1+\bar{\gamma}_N}{N} + \sup_{s\in[0,T]}|M_s^N| + C_L \int_0^T \epsilon_N(s) \d{s}.
\end{equation}
Because $\delta_N T + (1+\bar{\gamma}_N)/N + \sup_{s\in[0,T]}|M_s^N|$ is nondecreasing in $T$, it follows from Gr\"{o}nwall's lemma that
\begin{equation}
\epsilon_N(T)
\leq \Bigl( \delta_N T + \frac{1+\bar{\gamma}_N}{N} + \sup_{s\in[0,T]}|M_s^N| \Bigr) \e{ C_L T }.
\end{equation}
Using Minkowsky's inequality for $p \in [1,\infty)$, strict monotonicity of $\exp{( C_L T )}$ and $\delta_N T$, and the triangle inequality, we find that
\begin{equation}
\lVert \epsilon_N(T) \rVert_p
\le \Bigl( \delta_N T + \frac{1+\bar{\gamma}_N}{N} + \lVert \sup_{s \in [0,T]} |M_s^N| \rVert_p \Bigr) \e{ C_L T }.
\end{equation}
Finally, using Doob's martingale inequality \cite{steele_stochastic_2001} for $p > 1$, we obtain
\begin{equation}
\lVert \epsilon_N(T) \rVert_p
\le \Big( \delta_N T + \frac{1+\bar{\gamma}_N}{N} + \kappa_p \lVert M_T^N \rVert_p \Big) \e{ C_L T },
\end{equation}
completing the first part of the proof.

For $p = 2$, this inequality can be further simplified by computing the increasing process associated to the martingale. Note specifically that for $l \geq 0$ we have
\begin{equation}
\expectation{ (M_l)^2 }
= \expectation{ \langle M_l \rangle }
= \expectationBig{ \sum_{i=0}^{l} {\textrm{Var}}[ \gamma_N(Z_i) ] }
\label{eqn:L2_Experience_of_ML2}
\end{equation}
where
\begin{equation}
\variance{ \gamma_N(x) }
= \sum_{k=0}^{N-x-1} (k+1)^2 p_{x, x+k+1} - \Bigl( \sum_{k=0}^{N-x-1} (k+1) p_{x, x+k+1} \Bigr)^2
= \psi_N(x).
\label{eqn:L2_Variance_of_gamma_N}
\end{equation}
Therefore for the scaled martingale $M_t^N$, we find by combining \refEquation{eqn:L2_Experience_of_ML2} and \refEquation{eqn:L2_Variance_of_gamma_N} that for $t > 0$
\begin{equation}
\lVert M_t^N \rVert_2^{2}
= \expectation{ (M_{t}^N)^2 }
= \frac{ \expectation{ M^2_{ [ tN ] } } }{N^2}
= \frac{1}{N^2} \sum_{i=0}^{[tN]} \psi_N(Z_i)
\leq \frac{\bar{\psi}_N t}{N}.
\label{eqn:Bound_on_L2_norm_of_MtN}
\end{equation}
This completes the second part of the proof.
\end{proof}

\section{Proof of \refProposition{prop:Diffusion_approximation_with_error_bound}}
\label{appendix:Proof_of_proposition_Diffusion_approximation_with_error_bound}

\begin{proof}
We adapt the results of Kurtz which were derived for continuous time Markov jump processes.
For doing so, we can
replace the Poisson processes  involved in the construction of the jump processes by some random walks
that can be used to construct discrete time Markov chains.
We can then use exactly the same steps as in \cite{kurtz78},
by first comparing the original process $Z^N$ to a diffusion of the form
\begin{equation}
\tilde Z^N_t
= \frac{1}{N} \sum_{l \le N} l B_l (N \sum_0^t p_N(l,\tilde Z^N_s) \d{s} ),
\end{equation}
that is a sum of a finite number of scaled independent Brownian motions $B_l$.

Rewriting the inequalities in \cite[(3.6)]{kurtz78}, and using a random walk version of the approximation lemma of Koml\'{o}s--Major--Tusn\'{a}dy \cite{berkes2014}, we obtain
\begin{equation}
\expectationBig{ \sup_{t \le T} |\tilde Z^N_t - Z^N_t| }
\leq C_2 \frac{\log(N)}{N}.
\end{equation}
This leads using the results of \cite[Section~3]{kurtz78} to
\begin{equation}
\expectationBig{ \sup_{t \le T} | W^N_t- W_t | }
\leq C_3 \frac{\log(N)}{\sqrt{N}},
\end{equation}
which concludes the proof.
\end{proof}

\section{Proof of \refProposition{prop:Bound_on_L2_norm_of_TNstar_minus_Tstar}}
\label{appendix:Proof_of_proposition_Bound_on_L2_norm_of_TNstar_minus_Tstar}

\begin{proof}
Remark that if $| z(s) - Z_{s}^N | \le \delta / 2$ for all $s>0$, that then
\begin{equation}
\Bigl| \frac{T^*_N}{N} - T^* \Bigr|
\leq \Bigl| (z-\tfrac{1}{2}\delta)^{-1}(1) - (z+\tfrac{1}{2}\delta)^{-1}(1)\Bigr|
\leq |(T^*+\tfrac{1}{2}\delta) - (T^*-\tfrac{1}{2}\delta)|
= \delta.
\end{equation}
Here, the last inequality follows from the fact that $\dot{z}(s) = 1 + \gamma(z(s)) \geq 1$, since
\begin{align}
(z-\tfrac{1}{2}\delta)(T^{*}+\tfrac{1}{2}\delta)
&
= (z-\tfrac{1}{2}\delta)(z^{-1}(1)+\tfrac{1}{2}\delta)
\nonumber \\ &
\geq z(z^{-1}(1))+ \tfrac{1}{2}\delta - \tfrac{1}{2}\delta
= 1
=(z-\tfrac{1}{2}\delta)( (z-\tfrac{1}{2}\delta)^{-1}(1)).
\end{align}
Thus the first claim follows directly from the observation that the event
\begin{equation}
\Big\{ \Bigl| \frac{T^*_N}{N} - T^* \Bigr| \ge \delta \Big\}
\subseteq \big\{ | z(s) - Z_{s}^N | \geq \tfrac{1}{2}\delta \big\},
\label{eqn:Subset_statement_on_TNstar_and_Tstar_wrt_zs_and_ZsN}
\end{equation}
and then using (i) Markov's inequality \cite{steele_stochastic_2001}, and (ii) invoking \refProposition{prop:LLN_comparing_Znt_and_Zt}, so that
\begin{equation}
\probabilityBig{ \Bigl| \frac{T^*_N}{N} - T^* \Bigr| \geq \delta }
\eqcom{\ref{eqn:Subset_statement_on_TNstar_and_Tstar_wrt_zs_and_ZsN}}\leq \probabilityBig{ | z(s) - Z_{s}^N | \geq \tfrac{1}{2}\delta }
\eqcom{i}\leq \frac{2}{\delta} \expectation{ | z(s) - Z_{s}^N | }
\eqcom{ii}\leq \frac{2 \omega_N}{\delta}.
\label{eqn:Bound_on_probability_of_difference_TNstar_Tstar_within_proof}
\end{equation}

Now (i) using that $Z_{T_N^*} / N = z(T^*) = 1$ together with \refEquation{eq:ODE.disc} and \refEquation{eq:martingale_decomp_Zn}, and (ii) after expanding the integrals, we find that
\begin{align}
\label{eq:Tmart}
\frac{T^*_N}{N}- T^*
&
\eqcom{i}= \int_0^{T^*} \gamma(z(s)) \d{s} - \int_0^{\frac{T^*_N}{N}} \gamma_N(Z_{sN}) \d{s}
- \frac{M_{T_N^*}}{N}
\nonumber \\ &
\eqcom{ii}= \int_0^{ \frac{T_N^*}{N} \wedge T^* } ( \gamma(z(s)) - \gamma_N(Z_{sN}) ) \d{s} - \frac{M_{T_N^*}}{N}
\nonumber \\ &
\phantom{=} + \int_{ \frac{T_N^*}{N} \wedge T^* }^{  T^* } \gamma(z(s)) \d{s} - \int_{ \frac{T_N^*}{N} \wedge T^* }^{  \frac{T_N^*}{N} } \gamma_N(Z_{sN}) \d{s}.
\end{align}
Then taking the absolute value and using the triangle inequality, it follows that
\begin{align}
\Bigl| \frac{T^*_N}{N} - T^* \Bigr|
\leq & \int_0^{ \frac{T_N^*}{N} \wedge T^*} |\gamma(z(s)) - \gamma_N(Z_{sN})| \d{s} + |M_{T_N^*/N}^N|
\nonumber \\ &
+ \int_{ \frac{T_N^*}{N} \wedge T^*}^{ T^* } | \gamma(z(s))| \d{s} + \int_{ \frac{T_N^*}{N} \wedge T^* }^{ \frac{T_N^*}{N} } | \gamma_N(Z_{sN}) | \d{s}.
\end{align}
Approximating $\gamma_N$ by $\gamma$ via \refEquation{eq:def_gamma}, using Lipschitz continuity of $\gamma$, and recalling that $\max \{ T_N^* / N, \allowbreak T^* \} \allowbreak \leq 1$, we find that
\begin{equation}
\Bigl| \frac{T^*_N}{N} - T^* \Bigr|
\leq 2 C_L \sup_{s \le1 } | z(s) - Z_{s}^N | + 2 \delta_N + |M_{T_N^*/N}^N| + \int_{ \frac{T_N^*}{N} \wedge T^* }^{ \frac{T_N^*}{N} \vee T^* } | \gamma(z(s)) | \d{s}.
\end{equation}

The continuity of $\gamma(x)$ guarantees that there exist constants $C_1, \varepsilon > 0$ such that (i) $\gamma(z(s)) \leq 1 - \varepsilon$ for all $s \geq C_1$, and (ii) $C_1 < T^* - \delta$, provided that $\delta$ is sufficiently small. There are now two possible cases: either (a) $C_1 \leq T_N^* / N \wedge T^*$, or (b) $T_N^* / N \wedge T^* < C_1 < T_N^* / N \vee T^*$. For convenience, we first split the integral according to
\begin{equation}
\int_{\frac{T_N^*}{N} \wedge T^*}^{\frac{T_N^*}{N} \vee T^*} | \gamma(z(s)) | \d{s}
= \int_{\frac{T_N^*}{N} \wedge T^*}^{\frac{T_N^*}{N} \vee T^*} | \gamma(z(s)) | ( \indicator{ s < C_1 } + \indicator{ s \geq C_1 } ) \d{s}.
\end{equation}
Then splitting further into case (a), we have that
\begin{equation}
\int_{\frac{T_N^*}{N} \wedge T^*}^{\frac{T_N^*}{N} \vee T^*} | \gamma(z(s)) | \indicatorBig{ s < C_1, C_1 \leq \frac{T_N^*}{N} \wedge T^* } \d{s} = 0,
\end{equation}
and
\begin{align}
&
\int_{\frac{T_N^*}{N} \wedge T^*}^{\frac{T_N^*}{N} \vee T^*} | \gamma(z(s)) | \indicatorBig{ s \geq C_1, C_1 \leq \frac{T_N^*}{N} \wedge T^* } \d{s}
\\ &
\leq (1-\varepsilon) \Bigl| \frac{T_N^*}{N} - T^* \Bigr| \indicatorBig{ C_1 \leq \frac{T_N^*}{N} \wedge T^* }.
\nonumber
\end{align}
Next let $C_2$ be a constant such that $C_2 \geq \int_{ T_N^* / N \wedge T^* }^{C_1} | \gamma(z(s)) | \d{s}$. We can then, after splitting further into case (b), bound
\begin{align}
&
\int_{\frac{T_N^*}{N} \wedge T^*}^{\frac{T_N^*}{N} \vee T^*} | \gamma(z(s)) | \indicatorBig{ s < C_1, \frac{T_N^*}{N} \wedge T^* < C_1 < \frac{T_N^*}{N} \vee T^* } \d{s}
\\ &
= \int_{ \frac{T_N^*}{N} \wedge T^* }^{C_1} | \gamma(z(s)) | \d{s} \indicatorBig{ \frac{T_N^*}{N} \wedge T^* < C_1 < \frac{T_N^*}{N} \vee T^* }
\nonumber \\ &
\leq C_2 \indicatorBig{ \frac{T_N^*}{N} \wedge T^* < C_1 < \frac{T_N^*}{N} \vee T^* }
\leq C_2 \indicatorBig{ \frac{T_N^*}{N} < C_1 },
\nonumber
\end{align}
since if $\indicator{ T_N^* / N \wedge T^* < C_1 < T_N^* / N \vee T^* } = 1$, clearly $T_N^* / N \wedge T^* < C_1$. But by construction $C_1 < T^*$, so it must hold that $T_N^* / N < C_1$ and thus $\indicator{ T_N^* / N < C_1 } = 1$. Next, we bound
\begin{align}
&
\int_{\frac{T_N^*}{N} \wedge T^*}^{\frac{T_N^*}{N} \vee T^*} | \gamma(z(s)) | \indicatorBig{ s \geq C_1, \frac{T_N^*}{N} \wedge T^* < C_1 < \frac{T_N^*}{N} \vee T^* } \d{s}
\\ &
\leq (1-\varepsilon) \Bigl| \frac{T_N^*}{N} - T^* \Bigr| \indicatorBig{ \frac{T_N^*}{N} \wedge T^* < C_1 < \frac{T_N^*}{N} \vee T^* }.
\nonumber
\end{align}
Summarizing, there thus exists a constant $C_2$ such that
\begin{align}
\Bigl| \frac{T^*_N}{N} - T^* \Bigr|
&
\leq 2 C_L \sup_{s \le1 } | z(s) - Z_{s}^N | + 2 \delta_N + |M_{T_N^*/N}^N|
\label{eqn:Bound_on_difference_Tnstar_Tstar_pathwise}
\\ &
\phantom{\leq}
+ (1-\varepsilon) \Bigl| \frac{T_N^*}{N} - T^* \Bigr| + C_2 \indicatorBig{ \frac{T_N^*}{N} < C_1 }. \nonumber
\end{align}

Now recall that if $|z(s) - Z_s^N | \leq \delta / 2$, then $| T_N^* / N - T^* | \leq \delta$. Moreover then also $T_N^* / N \geq C_1$ since $C_1 < T^* - \delta$. Hence,
\begin{equation}
\Bigl\{ \frac{T_N^*}{N} < C_1 \Bigr\}
\subset \bigl\{ |z(s) - Z_s^N | > \tfrac{1}{2} \delta \bigr\}.
\label{eqn:Event_on_C1_implies_event_on_zs}
\end{equation}
Then by (i) collecting terms in and subsequently using \refEquation{eqn:Bound_on_difference_Tnstar_Tstar_pathwise}, and then (ii) applying Minkowski's inequality \cite{rudin_real_1987}, we obtain
\begin{align}
&
\varepsilon \Bigl\lVert \frac{T_N^*}{N} - T^* \Bigr\rVert_2
\eqcom{i}\leq \Bigl\lVert 2 C_L \sup_{s \le1 } | z(s) - Z_{s}^N | + 2 \delta_N + | M^N_{T_N^*/N} | + C_2 \indicatorBig{ \frac{T_N^*}{N} < C_1 } \Bigr\rVert_2
\nonumber \\ &
\eqcom{ii}\leq 2 C_L \lVert \sup_{s \le1 } | z(s) - Z_{s}^N | \rVert_2 + 2 \delta_N + \lVert M^N_{T_N^*/N} \rVert_2 + C_2 \Bigl\lVert \indicatorBig{ \frac{T_N^*}{N} < C_1 } \Bigr\rVert_2.
\end{align}
We now note that (iii) since $f(y) = y^2$ is monotonically increasing for $y \geq 0$ and (iv) by Markov's inequality,
\begin{align}
&
\Bigl\lVert \indicatorBig{ \frac{T_N^*}{N} < C_1 } \Bigr\rVert_2
= \probabilityBig{ \frac{T_N^*}{N} < C_1 }^{\frac{1}{2}}
\eqcom{\ref{eqn:Event_on_C1_implies_event_on_zs}}\leq \probability{ | z(s) - Z_s^N | > \tfrac{1}{2} \delta }^{\frac{1}{2}}
\\ &
\eqcom{iii}= \probability{ | z(s) - Z_s^N |^2 > \tfrac{1}{4} \delta^2 }^{\frac{1}{2}}
\eqcom{iv}\leq \frac{2}{\delta} \expectation{ | z(s) - Z_s^N |^2 }^{\frac{1}{2}}
= \frac{2}{\delta} \lVert z(s) - Z_s^N \rVert_2.
\nonumber
\end{align}
Therefore,
\begin{equation}
\varepsilon \Bigl\lVert \frac{T^*_N}{N} - T^* \Big\rVert_2
\leq \bigl( 2 C_L + \frac{2C_2}{\delta} \bigr) \lVert \sup_{s \le 1 } | z(s) - Z_{s}^N | \rVert_2 + 2\delta_N + \lVert M_{T_N^*/N}^N \rVert_2.
\end{equation}
Thus by finally using \refProposition{prop:LLN_comparing_Znt_and_Zt} and \refEquation{eqn:Bound_on_L2_norm_of_MtN}, we have that there exist constants $C_3$, $C_4$ so that
\begin{equation}
\varepsilon \Big\lVert \frac{T^*_N}{N} - T^* \Big\rVert_2
\leq C_3 \omega_N + 2\delta_N + \sqrt{ \frac{\bar{\psi}_N}{N} }
\leq C_4 \omega_N,
\end{equation}
which concludes the proof.
\end{proof}

\section{Proof of \refProposition{prop:CLT_of_Tstar}}
\label{appendix:Proof_of_proposition_CLT_of_Tstar}

\begin{proof}
First, recall that by \refEquation{eq:ODE.disc} and \refEquation{eq:martingale_decomp_Zn}, see \refEquation{eq:Tmart},
\begin{equation}
\frac{T^*_N}{N} - T^*
= \int_0^{T^*} \gamma(z(s)) \d{s} - \int_0^{\frac{T^*_N}{N}} \gamma_N(Z_{sN}) \d{s} - M_{T_N^*/N}^N,
\label{eqn:Difference_of_TstarN_over_N_and_Tstar}
\end{equation}
Note furthermore that
\begin{align}
W_{T^*}^N
&
\eqcom{\ref{eqn:Definition_of_WNt}}= \sqrt{N} \bigl( Z_{T^*}^N - z(T^*) \bigr)
\eqcom{\ref{eq:martingale_decomp_Zn}}= \sqrt{N} \Bigl( \int_0^{ \frac{[T^*N]}{N} } ( 1 + \gamma_N(Z_{sN}) ) \d{s} + M_{T^*}^N - z(T^*) \Bigr)
\nonumber \\ &
= \sqrt{N} \Bigl( \int_0^{T^*} ( 1 + \gamma_N(Z_{sN}) ) \d{s} + M_{T^*}^N - z(T^*) \Bigr) + \sqrt{N} \Delta_{N,T^*}
\label{eq:decomposition_Wt*}
\end{align}
where $M_{T^*}^N = M_{[T^*N]} / N$. Recall that the error $\Delta_{N,T^*}$ introduced by replacing the upper integration boundary, can readily be bounded by $|\Delta_{N,T^*}| \leq (1+\bar{\gamma}_N)/N$, see \refEquation{eqn:Bound_on_the_error_by_a_replacement_independent_of_t}.

Comparing \refEquation{eqn:Difference_of_TstarN_over_N_and_Tstar} and \refEquation{eq:decomposition_Wt*}, a subsequent natural series of steps would be to \textnormal{(i)} add comparison terms $\pm W_{T^*}^N$ and use the triangle inequality, and then \textnormal{(ii)} substitute \refEquation{eq:decomposition_Wt*}, use the triangle inequality, and upper bound $|\Delta_{N,T^*}| \leq (1+\bar{\gamma}_N)/N$, after which we arrive at
\begin{align}
&
\Bigl| \sqrt{N} \Bigl( \frac{T_N^*}{N} - T^* \Bigr) + W_{T^*} \Bigr|
\nonumber
\\ &
\overset{\textnormal{(i)}} \leq \Bigl| \sqrt{N} \Bigl( \int_0^{T^*} \gamma(z(s)) \d{s} - \int_0^{\frac{T^*_N}{N}} \gamma_N(Z_{sN}) \d{s} - M_{T_N^*/N}^N \Bigr) + W_{T^*}^N \Bigr|
+ | W_{T^*} - W_{T^*}^N |
\nonumber \\ &
\overset{\textnormal{(ii)}} \leq \Bigl| \sqrt{N} \Bigl( \int_0^{T^*} \gamma(z(s)) \d{s} - \int_0^{\frac{T^*_N}{N}} \gamma_N(Z_{sN}) \d{s} - M_{T_N^*/N}^N \Bigr)
\nonumber \\ &
\phantom{\leq \Big|} + \sqrt{N} \Bigl( \int_0^{T^*} ( 1 + \gamma_N(Z_{sN}) ) \d{s} + M_{T^*}^N - z(T^*) \Bigr) \Bigr| + | W_{T^*} - W_{T^*}^N | + \frac{1+\bar{\gamma}_N}{\sqrt{N}}
\nonumber \\ &
= \textnormal{ term I } + \textnormal{ term II } + \frac{1+\bar{\gamma}_N}{\sqrt{N}}.
\end{align}
We will now proceed and bound term I and II.

Note that the expectation of term II can be directly bounded by \refProposition{prop:Diffusion_approximation_with_error_bound}, i.e.\ there exists a constant $C_2$ such that
\begin{equation}
\expectation{ \textnormal{term II} }
= \expectation{ | W_{T^*}  - W_{T^*}^N | }
\leq \expectation{ \sup_{t \leq 1} | W_{t}  - W_t^N | }
\leq C_2 \frac{\log(N)}{\sqrt{N}}.
\end{equation}

Bounding term I requires more work. Using \refEquation{eq:decomposition_Wt*}, the integral version of $\dot{z} = 1 + \gamma(z)$, and the triangle inequality, we find that
\begin{align}
\textnormal{term I}
&
\leq \sqrt{N} \Bigl| \int_0^{T^*} \gamma_N(Z_{sN}) \d{s} - \int_0^{\frac{T^*_N}{N}} \gamma_N(Z_{sN}) \d{s} \Bigr| +
\sqrt{N} | M_{T^*}^N -M_{T_N^*/N}^N |
\nonumber \\ &
= \textnormal{term Ia} + \textnormal{term Ib},
\end{align}
and we now proceed with bounding term Ia and Ib separately.

In order to bound term Ia, we add comparison terms $\pm \gamma( Z_{sN} / N )$ and $\pm \gamma(z(s))$ and use the triangle inequality, so that\footnote{Note that we use the notation that $\int_a^b = - \int_b^a$ when $a > b$.}
\begin{align}
\textnormal{term Ia}
&
\leq \sqrt{N} \Bigl| \int_{\frac{T^*_N}{N}}^{T^*} \gamma_N(Z_{sN}) - \gamma \Bigl( \frac{Z_{sN}}{N} \Bigr) \d{s} \Bigr|
+ \sqrt{N} \Bigl| \int_{\frac{T^*_N}{N}}^{T^*} \gamma \Bigl( \frac{Z_{sN}}{N} \Bigr) -\gamma(z(s)) \d{s} \Bigr|
\nonumber \\ &
\phantom{\leq} + \sqrt{N} \Bigr| \int_{\frac{T^*_N}{N}}^{T^*} \gamma(z(s)) \d{s} \Bigl|
\end{align}
Then by approximating $\gamma_N$ by $\gamma$, using the Lipschitz continuity of $\gamma$, and upper bounding the first two integrands, we find that
\begin{align}
\textnormal{term Ia}
&
\leq \sqrt{N} \delta_N \Bigl| \frac{T_N^*}{N} - T^* \Bigr|
+ \sqrt{N}C_L \sup_{s \leq 1} \Bigl| \frac{Z_{sN}}{N} - z(s) \Bigr| \Bigl| \frac{T_N^*}{N} - T^* \Bigr|
\nonumber \\ &
\phantom{\leq} + \sqrt{N} \Bigl| \int_{\frac{T^*_N}{N}}^{T^*} \gamma(z(s)) \d{s} \Bigr|
\end{align}
Taking the expectation and using the triangle inequality, it follows that
\begin{align}
\expectation{ \textnormal{term Ia} }
&
\leq \sqrt{N}\delta_N \expectationBig{ \Bigl| \frac{T_N^*}{N} - T^* \Bigr| }
+ \sqrt{N} C_L \expectationBig{ \sup_{s \leq 1} \Bigl| \frac{Z_{sN}}{N} - z(s) \Bigr| \Bigl| \frac{T_N^*}{N} - T^* \Bigr| }
\nonumber \\ &
\phantom{\leq}
+ \sqrt{N} \expectationBig{ \int_{\frac{T^*_N}{N}}^{T^*} \left| \gamma(z(s)) \right| \d{s} }.
\end{align}
Applying H\"{o}lder's inequality \cite{rudin_real_1987},
\begin{align}
\expectation{ \textnormal{term Ia} }
&
\leq \sqrt{N}\delta_N \Big\| \frac{T_N^*}{N} - T^* \Big\|_2
+ \sqrt{N}C_L \Big\| \sup_{s \leq 1} \Bigl| \frac{Z_{sN}}{N} - z(s) \Bigr| \Big\|_2 \Big\| \frac{T_N^*}{N} - T^* \Big\|_2
\nonumber \\ &
\phantom{\leq} + \sqrt{N} \expectationBig{ \int_{\frac{T^*_N}{N}}^{T^*} \left| \gamma(z(s)) \right| \d{s} },
\end{align}
and finally Propositions~\ref{prop:LLN_comparing_Znt_and_Zt} and \ref{prop:Bound_on_L2_norm_of_TNstar_minus_Tstar}, we end up with
\begin{equation}
\expectation{ \textnormal{term Ia} }
\leq \Omega_N \sqrt{N} ( \delta_N + C_L \omega_N ) + \sqrt{N} \expectationBig{ \int_{\frac{T^*_N}{N}}^{T^*} \left| \gamma(z(s)) \right| \d{s} }.
\label{eqn:Intermediate_bound_on_expectation_of_term_Ia}
\end{equation}

In order to deal with the last term in \refEquation{eqn:Intermediate_bound_on_expectation_of_term_Ia}, we will use a Taylor expansion of order zero around $T^*$. Specifically, we write
\begin{equation}
\gamma(z(s))
= \gamma(z(T^*))+ c (s-T^*) + R_2
= c(s-T^*) + R_2,
\end{equation}
where we have recalled that $\gamma(1) = 0$ by assumption and $z(T^*) = 1$. Then, by (i) the triangle inequality, (ii) upper bounding the integrand, and (iii) evaluating the integral, there exists a constant $C_3$ so that
\begin{align}
&
\sqrt{N} \int_{\frac{T^*_N}{N}}^{T^*} | \gamma(z(s)) | \d{s}
\eqcom{i}\leq \sqrt{N} \int_{\frac{T^*_N}{N}}^{T^*}c | s - T^* | + | R_2 | \d{s}
\nonumber \\ &
\eqcom{ii}\leq \sqrt{N} \int_{\frac{T^*_N}{N}}^{T^*}c \Bigl| \frac{T^*_N}{N}-T^* \Bigr| + | R_2 | \d{s}
\eqcom{iii}\leq C_3 \sqrt{N} \Bigl| \frac{T^*_N}{N}-T^* \Bigr|^2,
\end{align}
where for the second term we have used that $| R_2 | = \bigO{ (s-T^*)^2 }$, and that $ | s - T^* | \leq 1$ for $s \in [ T^*_N / N, T^* ]$. Therefore, by \refProposition{prop:Bound_on_L2_norm_of_TNstar_minus_Tstar},
\begin{align}
\sqrt{N} \expectationBig{ \int_{\frac{T^*_N}{N}}^{T^*} | \gamma(z(s)) | \d{s} }
&
\leq C_3 \sqrt{N} \expectationBig{ \Bigl| \frac{T^*_N}{N} - T^* \Bigr|^2 }
\nonumber \\ &
\leq C_3 \sqrt{N} \Bigl\| \frac{T_N^*}{N} - T^* \Bigr\|_2 ^2
\leq C_3 \Omega_N^2 \sqrt{N}.
\label{eqn:Intermediate_2_bound_for_IA}
\end{align}
Ultimately bounding \refEquation{eqn:Intermediate_bound_on_expectation_of_term_Ia} using \refEquation{eqn:Intermediate_2_bound_for_IA}, we conclude that there exists a constant $C_1$ such that
\begin{equation}
\expectation{ \textnormal{term Ia} }
\leq \Omega_N \sqrt{N} ( \delta_N + C_L \omega_N + C_3 \Omega_N )
\leq C_1 \omega_N^2 \sqrt{N}.
\label{eq:bound_IA}
\end{equation}

To finish the proof we still need to bound the expectation of term Ib, that is, $\sqrt{N} \expectation{ | M_{T^*}^N - M_{T_N^*/N}^N | }$. By (i) Cauchy--Schwarz's inequality, (ii) definition of the scaled martingale, (iii) calculating the increasing process similar to \refEquation{eqn:L2_Experience_of_ML2}--\refEquation{eqn:Bound_on_L2_norm_of_MtN}, and (iv) $\expectation{ ( M_t - M_s )^2 } = \expectation{ M_t^2 } - \expectation{ M_s^2 }$ for $t > s$ as a consequence of $M_t$ being a martingale, we find
\begin{align}
&
\sqrt{N} \expectation{ | M_{T_N^*/N}^N - M_{T^*}^N | }
\eqcom{i}\leq \sqrt{N} \expectation{ | M_{T_N^*/N}^N - M_{T^*}^N |^2 }^{\frac{1}{2}}
\eqcom{ii}= \frac{1}{\sqrt{N}} \expectation{ | M_{T_N^*} - M_{[T^*N]} |^2 }^{\frac{1}{2}}
\nonumber \\ &
\eqcom{iii}= \frac{1}{\sqrt{N}} \expectation{ < M_{T_N^*} - M_{[T^*N]} > }^{\frac{1}{2}}
\eqcom{iv}= \expectationBig{ \frac{1}{N} \sum_{i = T_N^* \wedge [ T^* N ] }^{ T_N^* \vee [ T^* N ] } \psi_N(Z_i) }^{\frac{1}{2}}.
\end{align}
Then (v) upper bounding $\psi_N(Z_i) \leq \bar{\psi}_N$, (vi) adding compensation terms $\pm T^*$, applying the triangle inequality and upper bounding $|T^*N-[T^*N]| \leq 1$, it follows (vii) from \refProposition{prop:Bound_on_L2_norm_of_TNstar_minus_Tstar} that
\begin{align}
\expectation{ \textrm{term Ib} }
&
= \sqrt{N} \expectation{ | M_{T_N^*/N}^N - M_{T^*}^N | }
\eqcom{v}\leq \expectationBig{ \bar{\psi}_N \Bigl| \frac{T_N^*}{N} - \frac{[T^*N]}{N} \Bigr| }^{\frac{1}{2}}
\nonumber \\ &
\eqcom{vi}\leq \Bigl( \bar{\psi}_N \Bigl\lVert \frac{T_N^*}{N} - T^* \Bigr\rVert_1 + \frac{\bar{\psi}_N}{N} \Bigr)^{\frac{1}{2}}
\eqcom{vii}\leq \Bigl( \bar{\psi}_N \Omega_N + \frac{\bar{\psi}_N}{N} \Bigr)^{\frac{1}{2}}.
\end{align}
Finally, we combine all bounds, resulting in
\begin{align}
&
\expectationBig{ \Bigl| \sqrt{N} \Bigl( \frac{T_N^*}{N} - T^* \Bigr) + W_{T^*} \Bigr| }
\leq \expectation{ \textnormal{term Ia} } + \expectation{ \textnormal{term Ib} } + \expectation{ \textnormal{term II} } + \frac{1+\bar{\gamma}_N}{\sqrt{N}}
\nonumber \\ &
\leq C_1 \omega_N^2 \sqrt{N} + \Bigl( \bar{\psi}_N \Omega_N + \frac{\bar{\psi}_N}{N} \Bigr)^{\frac{1}{2}} + C_2 \frac{\log(N)}{\sqrt{N}} + \frac{1 + \bar{\gamma}_N}{\sqrt{N}}.
\label{eqn:Penultimate_bound_in_the_proof}
\end{align}

If the distribution of the number of neighbors is such that $\delta_N = \smallO{ 1 / \sqrt{N} }$, $\bar{\gamma}_N = \smallO{ \sqrt{N} }$ and $\bar{\psi}_N = \smallO{ N^{1/4} }$, then $\omega_N = \smallO{ 1 / N^{3/8} }$ and $\Omega_N = \smallO{ 1 / N^{3/8} }$, and all the product-terms in \refEquation{eqn:Penultimate_bound_in_the_proof} converge to $0$ as $N \to \infty$. We have thus proven that under these conditions, the limit is a Gaussian random variable with variance
\begin{equation}
\sigma^2 = \expectation{ W_{T^*}^2 }.
\end{equation}
Defining $m(t) = \expectation{ W_{t}^2 }$, and using It\^{o}'s formula \cite{steele_stochastic_2001}, note that
\begin{equation}
\expectation{ W_t^2 }
= \expectationBig{ \int_0^t 2 W_s d W_s + \frac{1}{2} 2 \beta_{t} }
= 2\int_0^{t} \gamma'(z(s)) \expectation{ W_s^2 } \d{s} + \beta(t),
\end{equation}
and hence $m(t)$ satisfies the differential system
\begin{equation}
\dot{m}
= -2 \dot{\gamma}(z(t)) m(t) + \dot{\beta},
\quad
\textrm{with}
\quad
m_0 = 0.
\end{equation}
This finishes the proof.
\end{proof}

\section{Proof of \refCorollary{cor:Limiting_behavior_of_the_spatial_fluid_limit_as_c_to_zero}}
\label{appendix:Proof_of_corollary_Limiting_behavior_of_the_spatial_fluid_limit_as_c_to_zero}

\begin{proof}
Consider the expansion $u(t) = t + \sum_{i = 1}^\infty c^i u_i(t)$. Substitute into \refEquation{eq:yt}, and Taylor expand the right-hand side to obtain
\begin{align}
1 + c u_1'(t) + \bigO{ c^2 }
&
= 1 + c \exp{ \Bigl( - \int_0^t \frac{ds}{ 1 - u(s) } \Bigr) }
\nonumber \\ &
= 1 + c \exp{ \Bigl( - \int_0^t \frac{1}{1-s} + \sum_{i=1}^\infty \frac{c^i u_i(s)}{ (1-s)^{i+1} } \d{s} \Bigr) }
\nonumber \\ &
= 1 + c (1-t) \exp{ \Bigl( \sum_{i=1}^\infty \frac{c^i u_i(s)}{ (1-s)^{i+1} } \d{s} \Bigr) }
= 1 + c (1-t) \bigl( 1 + \bigO{c} \bigr).
\end{align}
Comparing terms we find that $u_1'(t) = (1-t)$ with initial condition $u_1(0) = 0$, leading to the conclusion that $u_1(t) = t ( 1 - \frac{1}{2} t )$. Therefore,
\begin{equation}
u(t)
= t + c t ( 1 - \tfrac{1}{2} t ) + \bigO{c^2}
= (1+c) t - \tfrac{1}{2} c t^2 + \bigO{c^2}.
\end{equation}

Exactly the same expansion is obtained for $l(t)$ when applying the approach up to and including order $\bigO{c}$. Since $u(t)$ is an upper bound and $l(t)$ is a lower bound for the fluid limit $z(t)$ of the spatial process, and both bounds have the same asymptotic behavior as $c \downarrow 0$, this completes the proof.
\end{proof}

\end{document}

%% file: Figure__Scaled_process_and_lower_and_upper_bounds_of_fluid_limit.tex
\begin{tikzpicture}
\begin{axis}[
	width=0.9\columnwidth, height=0.618*0.9\columnwidth,
	xmin=0., xmax=0.9,
	ymin=0, ymax=1,
	xtick={0,0.2,...,1},
	every axis x label/.style={at={(ticklabel cs:0.5)},anchor=north},
	ytick={0,1},
	every axis y label/.style={at={(ticklabel cs:0.5)},rotate=90,anchor=south},
	scaled ticks=true,
	xlabel={$t$},
	]

\addplot[color=black, mark=none] plot coordinates {
(0.,0.) (0.00401606,0.00961859) (0.00803213,0.0192145) (0.0120482,0.0287876) (0.0160643,0.0383378) (0.0200803,0.0478649) 
(0.0240964,0.0573688) (0.0281124,0.0668494) (0.0321285,0.0763065) (0.0361446,0.08574) (0.0401606,0.0951498) (0.0441767,0.104536) 
(0.0481928,0.113898) (0.0522088,0.123235) (0.0562249,0.132549) (0.060241,0.141837) (0.064257,0.151102) (0.0682731,0.160341) 
(0.0722892,0.169556) (0.0763052,0.178745) (0.0803213,0.18791) (0.0843373,0.197048) (0.0883534,0.206162) (0.0923695,0.215249) 
(0.0963855,0.224311) (0.100402,0.233347) (0.104418,0.242356) (0.108434,0.251339) (0.11245,0.260296) (0.116466,0.269225) 
(0.120482,0.278128) (0.124498,0.287004) (0.128514,0.295852) (0.13253,0.304673) (0.136546,0.313466) (0.140562,0.322231) 
(0.144578,0.330969) (0.148594,0.339678) (0.15261,0.348358) (0.156627,0.35701) (0.160643,0.365633) (0.164659,0.374227) 
(0.168675,0.382792) (0.172691,0.391327) (0.176707,0.399832) (0.180723,0.408308) (0.184739,0.416753) (0.188755,0.425168) 
(0.192771,0.433553) (0.196787,0.441906) (0.200803,0.450228) (0.204819,0.458519) (0.208835,0.466779) (0.212851,0.475006) 
(0.216867,0.483202) (0.220884,0.491365) (0.2249,0.499496) (0.228916,0.507593) (0.232932,0.515658) (0.236948,0.523689) 
(0.240964,0.531686) (0.24498,0.539649) (0.248996,0.547578) (0.253012,0.555473) (0.257028,0.563333) (0.261044,0.571157) 
(0.26506,0.578946) (0.269076,0.586699) (0.273092,0.594416) (0.277108,0.602096) (0.281124,0.60974) (0.285141,0.617347) 
(0.289157,0.624916) (0.293173,0.632447) (0.297189,0.63994) (0.301205,0.647394) (0.305221,0.65481) (0.309237,0.662186) 
(0.313253,0.669523) (0.317269,0.676819) (0.321285,0.684075) (0.325301,0.69129) (0.329317,0.698463) (0.333333,0.705595) 
(0.337349,0.712685) (0.341365,0.719732) (0.345382,0.726735) (0.349398,0.733696) (0.353414,0.740612) (0.35743,0.747483) 
(0.361446,0.75431) (0.365462,0.761091) (0.369478,0.767826) (0.373494,0.774514) (0.37751,0.781155) (0.381526,0.787748) 
(0.385542,0.794293) (0.389558,0.800789) (0.393574,0.807236) (0.39759,0.813632) (0.401606,0.819978) (0.405622,0.826272) 
(0.409639,0.832515) (0.413655,0.838704) (0.417671,0.84484) (0.421687,0.850922) (0.425703,0.856949) (0.429719,0.86292) 
(0.433735,0.868835) (0.437751,0.874693) (0.441767,0.880492) (0.445783,0.886233) (0.449799,0.891913) (0.453815,0.897533) 
(0.457831,0.903092) (0.461847,0.908587) (0.465863,0.914019) (0.46988,0.919387) (0.473896,0.924688) (0.477912,0.929923) 
(0.481928,0.93509) (0.485944,0.940188) (0.48996,0.945216) (0.493976,0.950172) (0.497992,0.955055) (0.502008,0.959864) 
(0.506024,0.964598) (0.51004,0.969254) (0.514056,0.973832) (0.518072,0.97833) (0.522088,0.982746) (0.526104,0.987078) 
(0.53012,0.991325) (0.534137,0.995486) (0.538153,0.999557) (0.542169,1.) (0.546185,1.) (0.550201,1.) (0.554217,1.) (0.558233,1.) 
(0.562249,1.) (0.566265,1.) (0.570281,1.) (0.574297,1.) (0.578313,1.) (0.582329,1.) (0.586345,1.) (0.590361,1.) (0.594378,1.) 
(0.598394,1.) (0.60241,1.) (0.606426,1.) (0.610442,1.) (0.614458,1.) (0.618474,1.) (0.62249,1.) (0.626506,1.) (0.630522,1.) 
(0.634538,1.) (0.638554,1.) (0.64257,1.) (0.646586,1.) (0.650602,1.) (0.654618,1.) (0.658635,1.) (0.662651,1.) (0.666667,1.) 
(0.670683,1.) (0.674699,1.) (0.678715,1.) (0.682731,1.) (0.686747,1.) (0.690763,1.) (0.694779,1.) (0.698795,1.) (0.702811,1.) 
(0.706827,1.) (0.710843,1.) (0.714859,1.) (0.718876,1.) (0.722892,1.) (0.726908,1.) (0.730924,1.) (0.73494,1.) (0.738956,1.) 
(0.742972,1.) (0.746988,1.) (0.751004,1.) (0.75502,1.) (0.759036,1.) (0.763052,1.) (0.767068,1.) (0.771084,1.) (0.7751,1.) 
(0.779116,1.) (0.783133,1.) (0.787149,1.) (0.791165,1.) (0.795181,1.) (0.799197,1.) (0.803213,1.) (0.807229,1.) (0.811245,1.) 
(0.815261,1.) (0.819277,1.) (0.823293,1.) (0.827309,1.) (0.831325,1.) (0.835341,1.) (0.839357,1.) (0.843373,1.) (0.84739,1.) 
(0.851406,1.) (0.855422,1.) (0.859438,1.) (0.863454,1.) (0.86747,1.) (0.871486,1.) (0.875502,1.) (0.879518,1.) (0.883534,1.) 
(0.88755,1.) (0.891566,1.) (0.895582,1.) (0.899598,1.) (0.903614,1.) (0.907631,1.) (0.911647,1.) (0.915663,1.) (0.919679,1.) 
(0.923695,1.) (0.927711,1.) (0.931727,1.) (0.935743,1.) (0.939759,1.) (0.943775,1.) (0.947791,1.) (0.951807,1.) (0.955823,1.) 
(0.959839,1.) (0.963855,1.) (0.967871,1.) (0.971888,1.) (0.975904,1.) (0.97992,1.) (0.983936,1.) (0.987952,1.) (0.991968,1.) 
(0.995984,1.) (1.,1.) 
};

\addplot[color=black, thick, mark=.] plot coordinates {
(0.,0.) (0.0005,0.001) (0.001,0.002) (0.0015,0.003) (0.002,0.006) (0.0025,0.0065) (0.003,0.009) (0.0035,0.0095) (0.004,0.01) 
(0.0045,0.011) (0.005,0.0125) (0.0055,0.014) (0.006,0.015) (0.0065,0.0155) (0.007,0.0165) (0.0075,0.0175) (0.008,0.018) (0.0085,0.0185) 
(0.009,0.0195) (0.0095,0.0205) (0.01,0.022) (0.0105,0.0245) (0.011,0.0255) (0.0115,0.0265) (0.012,0.0275) (0.0125,0.028) (0.013,0.0295) 
(0.0135,0.0315) (0.014,0.033) (0.0145,0.034) (0.015,0.0345) (0.0155,0.035) (0.016,0.036) (0.0165,0.037) (0.017,0.0385) (0.0175,0.039) 
(0.018,0.0395) (0.0185,0.04) (0.019,0.0415) (0.0195,0.043) (0.02,0.0435) (0.0205,0.0445) (0.021,0.0455) (0.0215,0.047) (0.022,0.0475) 
(0.0225,0.0495) (0.023,0.05) (0.0235,0.0505) (0.024,0.052) (0.0245,0.053) (0.025,0.0545) (0.0255,0.0555) (0.026,0.056) (0.0265,0.0565) 
(0.027,0.0575) (0.0275,0.059) (0.028,0.0605) (0.0285,0.0625) (0.029,0.064) (0.0295,0.065) (0.03,0.0665) (0.0305,0.0675) (0.031,0.0685) 
(0.0315,0.0705) (0.032,0.0715) (0.0325,0.0725) (0.033,0.073) (0.0335,0.0745) (0.034,0.075) (0.0345,0.076) (0.035,0.077) (0.0355,0.078) 
(0.036,0.079) (0.0365,0.0805) (0.037,0.082) (0.0375,0.083) (0.038,0.085) (0.0385,0.0855) (0.039,0.0875) (0.0395,0.088) (0.04,0.089) 
(0.0405,0.0895) (0.041,0.0905) (0.0415,0.091) (0.042,0.0915) (0.0425,0.092) (0.043,0.093) (0.0435,0.094) (0.044,0.0955) (0.0445,0.0965) 
(0.045,0.0975) (0.0455,0.098) (0.046,0.099) (0.0465,0.101) (0.047,0.1015) (0.0475,0.103) (0.048,0.104) (0.0485,0.105) (0.049,0.106) 
(0.0495,0.108) (0.05,0.1095) (0.0505,0.11) (0.051,0.111) (0.0515,0.1115) (0.052,0.1125) (0.0525,0.1135) (0.053,0.1145) (0.0535,0.1155) 
(0.054,0.1165) (0.0545,0.118) (0.055,0.119) (0.0555,0.121) (0.056,0.1215) (0.0565,0.122) (0.057,0.123) (0.0575,0.1235) (0.058,0.1245) 
(0.0585,0.125) (0.059,0.1265) (0.0595,0.1275) (0.06,0.128) (0.0605,0.129) (0.061,0.1305) (0.0615,0.1315) (0.062,0.1325) (0.0625,0.1335) 
(0.063,0.1345) (0.0635,0.135) (0.064,0.1355) (0.0645,0.136) (0.065,0.137) (0.0655,0.1375) (0.066,0.138) (0.0665,0.1385) (0.067,0.14) 
(0.0675,0.1415) (0.068,0.142) (0.0685,0.1435) (0.069,0.144) (0.0695,0.145) (0.07,0.146) (0.0705,0.147) (0.071,0.1485) (0.0715,0.15) 
(0.072,0.1505) (0.0725,0.153) (0.073,0.1545) (0.0735,0.156) (0.074,0.157) (0.0745,0.1575) (0.075,0.1585) (0.0755,0.1595) (0.076,0.16) 
(0.0765,0.1615) (0.077,0.162) (0.0775,0.1625) (0.078,0.1635) (0.0785,0.164) (0.079,0.165) (0.0795,0.1655) (0.08,0.167) (0.0805,0.1675) 
(0.081,0.1685) (0.0815,0.169) (0.082,0.171) (0.0825,0.1715) (0.083,0.172) (0.0835,0.173) (0.084,0.1745) (0.0845,0.1755) (0.085,0.176) 
(0.0855,0.1775) (0.086,0.179) (0.0865,0.1795) (0.087,0.181) (0.0875,0.1815) (0.088,0.182) (0.0885,0.184) (0.089,0.186) (0.0895,0.1875) 
(0.09,0.188) (0.0905,0.1895) (0.091,0.1905) (0.0915,0.1915) (0.092,0.1925) (0.0925,0.1945) (0.093,0.197) (0.0935,0.198) (0.094,0.1985) 
(0.0945,0.199) (0.095,0.2005) (0.0955,0.2015) (0.096,0.2035) (0.0965,0.2055) (0.097,0.206) (0.0975,0.207) (0.098,0.2085) (0.0985,0.209) 
(0.099,0.21) (0.0995,0.2105) (0.1,0.211) (0.1005,0.2125) (0.101,0.2135) (0.1015,0.2145) (0.102,0.216) (0.1025,0.2165) (0.103,0.217) 
(0.1035,0.2175) (0.104,0.2195) (0.1045,0.2215) (0.105,0.223) (0.1055,0.2235) (0.106,0.2255) (0.1065,0.226) (0.107,0.227) 
(0.1075,0.2285) (0.108,0.2305) (0.1085,0.233) (0.109,0.235) (0.1095,0.236) (0.11,0.2375) (0.1105,0.238) (0.111,0.239) (0.1115,0.241) 
(0.112,0.242) (0.1125,0.2435) (0.113,0.244) (0.1135,0.2445) (0.114,0.245) (0.1145,0.246) (0.115,0.2465) (0.1155,0.247) (0.116,0.2475) 
(0.1165,0.25) (0.117,0.252) (0.1175,0.253) (0.118,0.254) (0.1185,0.255) (0.119,0.257) (0.1195,0.2575) (0.12,0.2585) (0.1205,0.26) 
(0.121,0.261) (0.1215,0.262) (0.122,0.2635) (0.1225,0.2645) (0.123,0.2655) (0.1235,0.2665) (0.124,0.267) (0.1245,0.2675) (0.125,0.2695) 
(0.1255,0.27) (0.126,0.2705) (0.1265,0.271) (0.127,0.2715) (0.1275,0.2725) (0.128,0.2735) (0.1285,0.274) (0.129,0.275) (0.1295,0.2765) 
(0.13,0.278) (0.1305,0.2785) (0.131,0.2795) (0.1315,0.2805) (0.132,0.2815) (0.1325,0.283) (0.133,0.2835) (0.1335,0.285) (0.134,0.286) 
(0.1345,0.2875) (0.135,0.2885) (0.1355,0.2895) (0.136,0.29) (0.1365,0.292) (0.137,0.2925) (0.1375,0.293) (0.138,0.294) (0.1385,0.2945) 
(0.139,0.2965) (0.1395,0.297) (0.14,0.2985) (0.1405,0.2995) (0.141,0.3) (0.1415,0.3015) (0.142,0.3025) (0.1425,0.3035) (0.143,0.304) 
(0.1435,0.3045) (0.144,0.3055) (0.1445,0.3065) (0.145,0.3075) (0.1455,0.3095) (0.146,0.311) (0.1465,0.312) (0.147,0.3135) 
(0.1475,0.3145) (0.148,0.3155) (0.1485,0.318) (0.149,0.319) (0.1495,0.3195) (0.15,0.32) (0.1505,0.321) (0.151,0.3225) (0.1515,0.3235) 
(0.152,0.324) (0.1525,0.3245) (0.153,0.3255) (0.1535,0.3265) (0.154,0.3275) (0.1545,0.3295) (0.155,0.3305) (0.1555,0.331) (0.156,0.332) 
(0.1565,0.3325) (0.157,0.333) (0.1575,0.3345) (0.158,0.335) (0.1585,0.3375) (0.159,0.3385) (0.1595,0.3395) (0.16,0.3405) 
(0.1605,0.3415) (0.161,0.342) (0.1615,0.3425) (0.162,0.343) (0.1625,0.344) (0.163,0.3445) (0.1635,0.345) (0.164,0.346) (0.1645,0.3465) 
(0.165,0.3475) (0.1655,0.35) (0.166,0.3505) (0.1665,0.351) (0.167,0.3515) (0.1675,0.352) (0.168,0.3525) (0.1685,0.353) (0.169,0.354) 
(0.1695,0.3545) (0.17,0.3555) (0.1705,0.3565) (0.171,0.3575) (0.1715,0.3585) (0.172,0.3605) (0.1725,0.361) (0.173,0.363) 
(0.1735,0.3635) (0.174,0.364) (0.1745,0.365) (0.175,0.367) (0.1755,0.3675) (0.176,0.369) (0.1765,0.3695) (0.177,0.37) (0.1775,0.3705) 
(0.178,0.3725) (0.1785,0.3735) (0.179,0.3755) (0.1795,0.376) (0.18,0.3775) (0.1805,0.3785) (0.181,0.3805) (0.1815,0.3815) 
(0.182,0.3835) (0.1825,0.384) (0.183,0.3845) (0.1835,0.386) (0.184,0.3875) (0.1845,0.3885) (0.185,0.3905) (0.1855,0.392) (0.186,0.3925) 
(0.1865,0.3935) (0.187,0.394) (0.1875,0.395) (0.188,0.3965) (0.1885,0.3975) (0.189,0.399) (0.1895,0.3995) (0.19,0.4015) (0.1905,0.403) 
(0.191,0.404) (0.1915,0.4055) (0.192,0.407) (0.1925,0.408) (0.193,0.409) (0.1935,0.4095) (0.194,0.4105) (0.1945,0.4115) (0.195,0.4125) 
(0.1955,0.413) (0.196,0.414) (0.1965,0.4155) (0.197,0.4165) (0.1975,0.418) (0.198,0.42) (0.1985,0.421) (0.199,0.422) (0.1995,0.423) 
(0.2,0.4235) (0.2005,0.424) (0.201,0.4245) (0.2015,0.426) (0.202,0.427) (0.2025,0.428) (0.203,0.4295) (0.2035,0.43) (0.204,0.431) 
(0.2045,0.432) (0.205,0.433) (0.2055,0.4335) (0.206,0.434) (0.2065,0.4345) (0.207,0.4365) (0.2075,0.4375) (0.208,0.4385) (0.2085,0.439) 
(0.209,0.4405) (0.2095,0.4415) (0.21,0.4425) (0.2105,0.444) (0.211,0.445) (0.2115,0.4455) (0.212,0.446) (0.2125,0.447) (0.213,0.448) 
(0.2135,0.4485) (0.214,0.4495) (0.2145,0.45) (0.215,0.4505) (0.2155,0.451) (0.216,0.4515) (0.2165,0.4525) (0.217,0.455) (0.2175,0.456) 
(0.218,0.4565) (0.2185,0.4575) (0.219,0.458) (0.2195,0.46) (0.22,0.4605) (0.2205,0.4615) (0.221,0.462) (0.2215,0.4635) (0.222,0.4645) 
(0.2225,0.466) (0.223,0.4665) (0.2235,0.467) (0.224,0.4675) (0.2245,0.4685) (0.225,0.4695) (0.2255,0.47) (0.226,0.4705) (0.2265,0.4715) 
(0.227,0.4725) (0.2275,0.473) (0.228,0.4735) (0.2285,0.474) (0.229,0.4755) (0.2295,0.476) (0.23,0.4765) (0.2305,0.477) (0.231,0.4775) 
(0.2315,0.478) (0.232,0.4805) (0.2325,0.4815) (0.233,0.483) (0.2335,0.484) (0.234,0.4845) (0.2345,0.487) (0.235,0.488) (0.2355,0.4895) 
(0.236,0.491) (0.2365,0.4915) (0.237,0.4925) (0.2375,0.494) (0.238,0.4955) (0.2385,0.496) (0.239,0.4965) (0.2395,0.4975) (0.24,0.4995) 
(0.2405,0.5005) (0.241,0.501) (0.2415,0.502) (0.242,0.5025) (0.2425,0.5035) (0.243,0.504) (0.2435,0.505) (0.244,0.5065) (0.2445,0.507) 
(0.245,0.5075) (0.2455,0.508) (0.246,0.5095) (0.2465,0.511) (0.247,0.512) (0.2475,0.5125) (0.248,0.513) (0.2485,0.514) (0.249,0.5145) 
(0.2495,0.515) (0.25,0.5155) (0.2505,0.5165) (0.251,0.517) (0.2515,0.5185) (0.252,0.519) (0.2525,0.52) (0.253,0.5205) (0.2535,0.5215) 
(0.254,0.523) (0.2545,0.5235) (0.255,0.526) (0.2555,0.527) (0.256,0.5285) (0.2565,0.529) (0.257,0.53) (0.2575,0.5305) (0.258,0.5325) 
(0.2585,0.533) (0.259,0.5335) (0.2595,0.535) (0.26,0.5355) (0.2605,0.537) (0.261,0.538) (0.2615,0.539) (0.262,0.5395) (0.2625,0.541) 
(0.263,0.5415) (0.2635,0.5425) (0.264,0.5435) (0.2645,0.5445) (0.265,0.545) (0.2655,0.5465) (0.266,0.5485) (0.2665,0.549) (0.267,0.55) 
(0.2675,0.551) (0.268,0.5515) (0.2685,0.5525) (0.269,0.553) (0.2695,0.5535) (0.27,0.554) (0.2705,0.5545) (0.271,0.555) (0.2715,0.5555) 
(0.272,0.5565) (0.2725,0.557) (0.273,0.5575) (0.2735,0.558) (0.274,0.5585) (0.2745,0.5595) (0.275,0.5615) (0.2755,0.562) (0.276,0.5625) 
(0.2765,0.5635) (0.277,0.564) (0.2775,0.565) (0.278,0.566) (0.2785,0.5675) (0.279,0.5685) (0.2795,0.569) (0.28,0.57) (0.2805,0.5705) 
(0.281,0.5715) (0.2815,0.5725) (0.282,0.573) (0.2825,0.5735) (0.283,0.574) (0.2835,0.5745) (0.284,0.575) (0.2845,0.5755) (0.285,0.5765) 
(0.2855,0.5775) (0.286,0.578) (0.2865,0.579) (0.287,0.58) (0.2875,0.5805) (0.288,0.581) (0.2885,0.582) (0.289,0.583) (0.2895,0.5835) 
(0.29,0.584) (0.2905,0.585) (0.291,0.5865) (0.2915,0.588) (0.292,0.589) (0.2925,0.59) (0.293,0.5905) (0.2935,0.591) (0.294,0.592) 
(0.2945,0.5925) (0.295,0.5935) (0.2955,0.595) (0.296,0.596) (0.2965,0.597) (0.297,0.598) (0.2975,0.5985) (0.298,0.5995) (0.2985,0.6) 
(0.299,0.6015) (0.2995,0.6025) (0.3,0.6035) (0.3005,0.604) (0.301,0.606) (0.3015,0.6065) (0.302,0.607) (0.3025,0.6075) (0.303,0.608) 
(0.3035,0.609) (0.304,0.61) (0.3045,0.611) (0.305,0.6115) (0.3055,0.6125) (0.306,0.6135) (0.3065,0.615) (0.307,0.6155) (0.3075,0.616) 
(0.308,0.617) (0.3085,0.6185) (0.309,0.6195) (0.3095,0.622) (0.31,0.623) (0.3105,0.624) (0.311,0.6245) (0.3115,0.626) (0.312,0.6265) 
(0.3125,0.6275) (0.313,0.6285) (0.3135,0.629) (0.314,0.63) (0.3145,0.631) (0.315,0.632) (0.3155,0.633) (0.316,0.6335) (0.3165,0.634) 
(0.317,0.635) (0.3175,0.6355) (0.318,0.636) (0.3185,0.6365) (0.319,0.637) (0.3195,0.6375) (0.32,0.638) (0.3205,0.639) (0.321,0.6395) 
(0.3215,0.641) (0.322,0.642) (0.3225,0.6425) (0.323,0.643) (0.3235,0.6435) (0.324,0.644) (0.3245,0.645) (0.325,0.646) (0.3255,0.6465) 
(0.326,0.6475) (0.3265,0.6485) (0.327,0.6495) (0.3275,0.6515) (0.328,0.652) (0.3285,0.653) (0.329,0.654) (0.3295,0.6555) (0.33,0.656) 
(0.3305,0.6565) (0.331,0.657) (0.3315,0.658) (0.332,0.659) (0.3325,0.66) (0.333,0.6605) (0.3335,0.661) (0.334,0.6615) (0.3345,0.662) 
(0.335,0.6625) (0.3355,0.664) (0.336,0.6645) (0.3365,0.665) (0.337,0.666) (0.3375,0.6665) (0.338,0.668) (0.3385,0.669) (0.339,0.67) 
(0.3395,0.6715) (0.34,0.6725) (0.3405,0.6735) (0.341,0.6745) (0.3415,0.6755) (0.342,0.676) (0.3425,0.677) (0.343,0.678) (0.3435,0.6785) 
(0.344,0.6795) (0.3445,0.68) (0.345,0.6805) (0.3455,0.6815) (0.346,0.682) (0.3465,0.6835) (0.347,0.684) (0.3475,0.6855) (0.348,0.686) 
(0.3485,0.6865) (0.349,0.687) (0.3495,0.6875) (0.35,0.689) (0.3505,0.6895) (0.351,0.6915) (0.3515,0.692) (0.352,0.693) (0.3525,0.6945) 
(0.353,0.695) (0.3535,0.696) (0.354,0.697) (0.3545,0.698) (0.355,0.699) (0.3555,0.6995) (0.356,0.7) (0.3565,0.7005) (0.357,0.7015) 
(0.3575,0.7025) (0.358,0.7035) (0.3585,0.7045) (0.359,0.705) (0.3595,0.7055) (0.36,0.706) (0.3605,0.707) (0.361,0.7075) (0.3615,0.7085) 
(0.362,0.709) (0.3625,0.7095) (0.363,0.711) (0.3635,0.712) (0.364,0.7125) (0.3645,0.7135) (0.365,0.7145) (0.3655,0.7155) (0.366,0.716) 
(0.3665,0.717) (0.367,0.7175) (0.3675,0.718) (0.368,0.719) (0.3685,0.72) (0.369,0.7205) (0.3695,0.721) (0.37,0.722) (0.3705,0.7225) 
(0.371,0.7235) (0.3715,0.7245) (0.372,0.7255) (0.3725,0.727) (0.373,0.7275) (0.3735,0.7285) (0.374,0.7295) (0.3745,0.73) (0.375,0.7305) 
(0.3755,0.731) (0.376,0.7315) (0.3765,0.732) (0.377,0.7325) (0.3775,0.733) (0.378,0.734) (0.3785,0.735) (0.379,0.736) (0.3795,0.737) 
(0.38,0.7375) (0.3805,0.738) (0.381,0.7385) (0.3815,0.739) (0.382,0.7395) (0.3825,0.741) (0.383,0.7425) (0.3835,0.743) (0.384,0.744) 
(0.3845,0.745) (0.385,0.746) (0.3855,0.7465) (0.386,0.747) (0.3865,0.7475) (0.387,0.748) (0.3875,0.7485) (0.388,0.749) (0.3885,0.7495) 
(0.389,0.75) (0.3895,0.7505) (0.39,0.751) (0.3905,0.752) (0.391,0.753) (0.3915,0.7535) (0.392,0.7545) (0.3925,0.7555) (0.393,0.756) 
(0.3935,0.757) (0.394,0.7585) (0.3945,0.7595) (0.395,0.76) (0.3955,0.761) (0.396,0.762) (0.3965,0.763) (0.397,0.764) (0.3975,0.765) 
(0.398,0.7655) (0.3985,0.766) (0.399,0.7665) (0.3995,0.7675) (0.4,0.7685) (0.4005,0.769) (0.401,0.7695) (0.4015,0.77) (0.402,0.7705) 
(0.4025,0.771) (0.403,0.7715) (0.4035,0.772) (0.404,0.7725) (0.4045,0.774) (0.405,0.7745) (0.4055,0.776) (0.406,0.7765) (0.4065,0.778) 
(0.407,0.7785) (0.4075,0.779) (0.408,0.7795) (0.4085,0.78) (0.409,0.7805) (0.4095,0.782) (0.41,0.7825) (0.4105,0.7835) (0.411,0.785) 
(0.4115,0.7855) (0.412,0.786) (0.4125,0.7865) (0.413,0.787) (0.4135,0.7875) (0.414,0.789) (0.4145,0.7895) (0.415,0.79) (0.4155,0.7905) 
(0.416,0.791) (0.4165,0.7915) (0.417,0.792) (0.4175,0.7925) (0.418,0.794) (0.4185,0.795) (0.419,0.7955) (0.4195,0.796) (0.42,0.7965) 
(0.4205,0.797) (0.421,0.7975) (0.4215,0.798) (0.422,0.799) (0.4225,0.8) (0.423,0.8005) (0.4235,0.8015) (0.424,0.802) (0.4245,0.803) 
(0.425,0.804) (0.4255,0.8055) (0.426,0.806) (0.4265,0.8065) (0.427,0.8085) (0.4275,0.809) (0.428,0.8095) (0.4285,0.81) (0.429,0.811) 
(0.4295,0.8115) (0.43,0.812) (0.4305,0.8125) (0.431,0.813) (0.4315,0.8135) (0.432,0.815) (0.4325,0.816) (0.433,0.8175) (0.4335,0.818) 
(0.434,0.8185) (0.4345,0.819) (0.435,0.8195) (0.4355,0.82) (0.436,0.821) (0.4365,0.8215) (0.437,0.822) (0.4375,0.8235) (0.438,0.824) 
(0.4385,0.8245) (0.439,0.825) (0.4395,0.8255) (0.44,0.826) (0.4405,0.827) (0.441,0.8275) (0.4415,0.828) (0.442,0.8285) (0.4425,0.829) 
(0.443,0.8295) (0.4435,0.83) (0.444,0.8305) (0.4445,0.831) (0.445,0.8315) (0.4455,0.832) (0.446,0.8325) (0.4465,0.833) (0.447,0.834) 
(0.4475,0.8345) (0.448,0.835) (0.4485,0.836) (0.449,0.8365) (0.4495,0.837) (0.45,0.8385) (0.4505,0.839) (0.451,0.8395) (0.4515,0.8405) 
(0.452,0.841) (0.4525,0.8415) (0.453,0.842) (0.4535,0.8425) (0.454,0.843) (0.4545,0.8435) (0.455,0.844) (0.4555,0.845) (0.456,0.8455) 
(0.4565,0.8465) (0.457,0.847) (0.4575,0.8475) (0.458,0.848) (0.4585,0.8485) (0.459,0.8495) (0.4595,0.85) (0.46,0.851) (0.4605,0.852) 
(0.461,0.853) (0.4615,0.854) (0.462,0.855) (0.4625,0.8555) (0.463,0.856) (0.4635,0.8565) (0.464,0.857) (0.4645,0.8575) (0.465,0.858) 
(0.4655,0.8585) (0.466,0.859) (0.4665,0.8595) (0.467,0.8605) (0.4675,0.861) (0.468,0.8615) (0.4685,0.862) (0.469,0.8625) (0.4695,0.864) 
(0.47,0.8655) (0.4705,0.866) (0.471,0.8665) (0.4715,0.867) (0.472,0.8675) (0.4725,0.868) (0.473,0.8685) (0.4735,0.869) (0.474,0.8695) 
(0.4745,0.87) (0.475,0.8705) (0.4755,0.871) (0.476,0.8715) (0.4765,0.872) (0.477,0.8725) (0.4775,0.873) (0.478,0.8735) (0.4785,0.874) 
(0.479,0.8745) (0.4795,0.875) (0.48,0.876) (0.4805,0.8765) (0.481,0.8775) (0.4815,0.878) (0.482,0.8785) (0.4825,0.879) (0.483,0.8795) 
(0.4835,0.88) (0.484,0.8805) (0.4845,0.881) (0.485,0.8815) (0.4855,0.882) (0.486,0.8825) (0.4865,0.883) (0.487,0.884) (0.4875,0.8845) 
(0.488,0.885) (0.4885,0.8855) (0.489,0.8865) (0.4895,0.887) (0.49,0.8875) (0.4905,0.888) (0.491,0.889) (0.4915,0.8895) (0.492,0.89) 
(0.4925,0.8905) (0.493,0.891) (0.4935,0.8915) (0.494,0.892) (0.4945,0.893) (0.495,0.8935) (0.4955,0.8945) (0.496,0.8955) (0.4965,0.896) 
(0.497,0.897) (0.4975,0.8975) (0.498,0.899) (0.4985,0.8995) (0.499,0.9) (0.4995,0.901) (0.5,0.902) (0.5005,0.903) (0.501,0.9035) 
(0.5015,0.904) (0.502,0.905) (0.5025,0.9055) (0.503,0.906) (0.5035,0.9065) (0.504,0.907) (0.5045,0.9075) (0.505,0.908) (0.5055,0.9085) 
(0.506,0.909) (0.5065,0.9095) (0.507,0.9105) (0.5075,0.911) (0.508,0.9115) (0.5085,0.912) (0.509,0.9125) (0.5095,0.913) (0.51,0.9135) 
(0.5105,0.914) (0.511,0.9145) (0.5115,0.915) (0.512,0.9165) (0.5125,0.917) (0.513,0.9175) (0.5135,0.918) (0.514,0.9185) (0.5145,0.919) 
(0.515,0.9195) (0.5155,0.92) (0.516,0.9205) (0.5165,0.921) (0.517,0.9215) (0.5175,0.922) (0.518,0.9225) (0.5185,0.923) (0.519,0.9235) 
(0.5195,0.9245) (0.52,0.925) (0.5205,0.9255) (0.521,0.926) (0.5215,0.927) (0.522,0.928) (0.5225,0.9285) (0.523,0.929) (0.5235,0.9295) 
(0.524,0.93) (0.5245,0.9305) (0.525,0.931) (0.5255,0.932) (0.526,0.9335) (0.5265,0.934) (0.527,0.9345) (0.5275,0.935) (0.528,0.9355) 
(0.5285,0.936) (0.529,0.9365) (0.5295,0.937) (0.53,0.9375) (0.5305,0.938) (0.531,0.9385) (0.5315,0.939) (0.532,0.9395) (0.5325,0.9405) 
(0.533,0.941) (0.5335,0.9415) (0.534,0.943) (0.5345,0.9435) (0.535,0.944) (0.5355,0.9445) (0.536,0.945) (0.5365,0.9455) (0.537,0.946) 
(0.5375,0.9465) (0.538,0.947) (0.5385,0.9475) (0.539,0.948) (0.5395,0.9485) (0.54,0.949) (0.5405,0.9495) (0.541,0.95) (0.5415,0.9505) 
(0.542,0.951) (0.5425,0.9515) (0.543,0.9525) (0.5435,0.953) (0.544,0.9535) (0.5445,0.9545) (0.545,0.955) (0.5455,0.9555) (0.546,0.956) 
(0.5465,0.9565) (0.547,0.957) (0.5475,0.9575) (0.548,0.9585) (0.5485,0.959) (0.549,0.9595) (0.5495,0.96) (0.55,0.9605) (0.5505,0.961) 
(0.551,0.9615) (0.5515,0.962) (0.552,0.9625) (0.5525,0.963) (0.553,0.9635) (0.5535,0.964) (0.554,0.9645) (0.5545,0.965) (0.555,0.9655) 
(0.5555,0.9665) (0.556,0.967) (0.5565,0.968) (0.557,0.9685) (0.5575,0.969) (0.558,0.9695) (0.5585,0.97) (0.559,0.9705) (0.5595,0.971) 
(0.56,0.9715) (0.5605,0.972) (0.561,0.9725) (0.5615,0.973) (0.562,0.9735) (0.5625,0.974) (0.563,0.9745) (0.5635,0.975) (0.564,0.9755) 
(0.5645,0.976) (0.565,0.9765) (0.5655,0.977) (0.566,0.9775) (0.5665,0.978) (0.567,0.9785) (0.5675,0.979) (0.568,0.9795) (0.5685,0.98) 
(0.569,0.981) (0.5695,0.9815) (0.57,0.982) (0.5705,0.9825) (0.571,0.983) (0.5715,0.9835) (0.572,0.984) (0.5725,0.9845) (0.573,0.985) 
(0.5735,0.9855) (0.574,0.986) (0.5745,0.9865) (0.575,0.987) (0.5755,0.9875) (0.576,0.988) (0.5765,0.9885) (0.577,0.989) (0.5775,0.9895) 
(0.578,0.99) (0.5785,0.9905) (0.579,0.991) (0.5795,0.9915) (0.58,0.992) (0.5805,0.9925) (0.581,0.993) (0.5815,0.9935) (0.582,0.994) 
(0.5825,0.9945) (0.583,0.995) (0.5835,0.9955) (0.584,0.996) (0.5845,0.9965) (0.585,0.997) (0.5855,0.9975) (0.586,0.998) (0.5865,0.9985) 
(0.587,0.999) (0.5875,0.9995) (0.588,1.) 
};

\addplot[color=black, mark=none] plot coordinates {
(0.,0.) (0.00401606,0.00957126) (0.00803213,0.019025) (0.0120482,0.0283609) (0.0160643,0.0375785) (0.0200803,0.0466776) 
(0.0240964,0.0556579) (0.0281124,0.064519) (0.0321285,0.0732608) (0.0361446,0.0818832) (0.0401606,0.0903859) (0.0441767,0.0987688) 
(0.0481928,0.107032) (0.0522088,0.115175) (0.0562249,0.123199) (0.060241,0.131103) (0.064257,0.138887) (0.0682731,0.146552) 
(0.0722892,0.154098) (0.0763052,0.161524) (0.0803213,0.168832) (0.0843373,0.176022) (0.0883534,0.183093) (0.0923695,0.190047) 
(0.0963855,0.196884) (0.100402,0.203605) (0.104418,0.210209) (0.108434,0.216699) (0.11245,0.223074) (0.116466,0.229335) 
(0.120482,0.235482) (0.124498,0.241518) (0.128514,0.247443) (0.13253,0.253256) (0.136546,0.258961) (0.140562,0.264557) 
(0.144578,0.270046) (0.148594,0.275428) (0.15261,0.280706) (0.156627,0.285879) (0.160643,0.29095) (0.164659,0.295919) 
(0.168675,0.300789) (0.172691,0.30556) (0.176707,0.310233) (0.180723,0.31481) (0.184739,0.319293) (0.188755,0.323683) 
(0.192771,0.327982) (0.196787,0.33219) (0.200803,0.33631) (0.204819,0.340348) (0.208835,0.344364) (0.212851,0.34838) 
(0.216867,0.352396) (0.220884,0.356412) (0.2249,0.360428) (0.228916,0.364444) (0.232932,0.36846) (0.236948,0.372476) 
(0.240964,0.376492) (0.24498,0.380508) (0.248996,0.384524) (0.253012,0.38854) (0.257028,0.392556) (0.261044,0.396572) 
(0.26506,0.400589) (0.269076,0.404605) (0.273092,0.408621) (0.277108,0.412637) (0.281124,0.416653) (0.285141,0.420669) 
(0.289157,0.424685) (0.293173,0.428701) (0.297189,0.432717) (0.301205,0.436733) (0.305221,0.440749) (0.309237,0.444765) 
(0.313253,0.448781) (0.317269,0.452797) (0.321285,0.456813) (0.325301,0.460829) (0.329317,0.464846) (0.333333,0.468862) 
(0.337349,0.472878) (0.341365,0.476894) (0.345382,0.48091) (0.349398,0.484926) (0.353414,0.488942) (0.35743,0.492958) 
(0.361446,0.496974) (0.365462,0.50099) (0.369478,0.505006) (0.373494,0.509022) (0.37751,0.513038) (0.381526,0.517054) 
(0.385542,0.52107) (0.389558,0.525087) (0.393574,0.529103) (0.39759,0.533119) (0.401606,0.537135) (0.405622,0.541151) 
(0.409639,0.545167) (0.413655,0.549183) (0.417671,0.553199) (0.421687,0.557215) (0.425703,0.561231) (0.429719,0.565247) 
(0.433735,0.569263) (0.437751,0.573279) (0.441767,0.577295) (0.445783,0.581311) (0.449799,0.585327) (0.453815,0.589344) 
(0.457831,0.59336) (0.461847,0.597376) (0.465863,0.601392) (0.46988,0.605408) (0.473896,0.609424) (0.477912,0.61344) 
(0.481928,0.617456) (0.485944,0.621472) (0.48996,0.625488) (0.493976,0.629504) (0.497992,0.63352) (0.502008,0.637536) 
(0.506024,0.641552) (0.51004,0.645568) (0.514056,0.649584) (0.518072,0.653601) (0.522088,0.657617) (0.526104,0.661633) 
(0.53012,0.665649) (0.534137,0.669665) (0.538153,0.673681) (0.542169,0.677697) (0.546185,0.681713) (0.550201,0.685729) 
(0.554217,0.689745) (0.558233,0.693761) (0.562249,0.697777) (0.566265,0.701793) (0.570281,0.705809) (0.574297,0.709825) 
(0.578313,0.713842) (0.582329,0.717858) (0.586345,0.721874) (0.590361,0.72589) (0.594378,0.729906) (0.598394,0.733922) 
(0.60241,0.737938) (0.606426,0.741954) (0.610442,0.74597) (0.614458,0.749986) (0.618474,0.754002) (0.62249,0.758018) 
(0.626506,0.762034) (0.630522,0.76605) (0.634538,0.770066) (0.638554,0.774082) (0.64257,0.778099) (0.646586,0.782115) 
(0.650602,0.786131) (0.654618,0.790147) (0.658635,0.794163) (0.662651,0.798179) (0.666667,0.802195) (0.670683,0.806211) 
(0.674699,0.810227) (0.678715,0.814243) (0.682731,0.818259) (0.686747,0.822275) (0.690763,0.826291) (0.694779,0.830307) 
(0.698795,0.834323) (0.702811,0.83834) (0.706827,0.842356) (0.710843,0.846372) (0.714859,0.850388) (0.718876,0.854404) 
(0.722892,0.85842) (0.726908,0.862436) (0.730924,0.866452) (0.73494,0.870468) (0.738956,0.874484) (0.742972,0.8785) (0.746988,0.882516) 
(0.751004,0.886532) (0.75502,0.890548) (0.759036,0.894564) (0.763052,0.89858) (0.767068,0.902597) (0.771084,0.906613) (0.7751,0.910629) 
(0.779116,0.914645) (0.783133,0.918661) (0.787149,0.922677) (0.791165,0.926693) (0.795181,0.930709) (0.799197,0.934725) 
(0.803213,0.938741) (0.807229,0.942757) (0.811245,0.946773) (0.815261,0.950789) (0.819277,0.954805) (0.823293,0.958821) 
(0.827309,0.962838) (0.831325,0.966854) (0.835341,0.97087) (0.839357,0.974886) (0.843373,0.978902) (0.84739,0.982918) 
(0.851406,0.986934) (0.855422,0.99095) (0.859438,0.994966) (0.863454,0.998982) (0.86747,1.) (0.871486,1.) (0.875502,1.) (0.879518,1.) 
(0.883534,1.) (0.88755,1.) (0.891566,1.) (0.895582,1.) (0.899598,1.) (0.903614,1.) (0.907631,1.) (0.911647,1.) (0.915663,1.) 
(0.919679,1.) (0.923695,1.) (0.927711,1.) (0.931727,1.) (0.935743,1.) (0.939759,1.) (0.943775,1.) (0.947791,1.) (0.951807,1.) 
(0.955823,1.) (0.959839,1.) (0.963855,1.) (0.967871,1.) (0.971888,1.) (0.975904,1.) (0.97992,1.) (0.983936,1.) (0.987952,1.) 
(0.991968,1.) (0.995984,1.) (1.,1.) 
};

\draw (0.325,0.8) node [fill=none, draw=none, text=black, opacity=1, text opacity=1] {\scriptsize $u(t)$};
\draw (0.525,0.85) node [fill=none, draw=none, text=black, opacity=1, text opacity=1] {\scriptsize $Z^N_t$};
\draw (0.425,0.5) node [fill=none, draw=none, text=black, opacity=1, text opacity=1] {\scriptsize $l(t)$};
\end{axis}
\end{tikzpicture}

%% file: Figure__Jamming_constant_as_function_of_c.tex
\begin{tikzpicture}
\begin{axis}[
	width=0.9\columnwidth, height=0.618*0.9\columnwidth,
	xmin=0., xmax=2.,
	ymin=0, ymax=1,
	xtick={0,0.25,...,4},
	every axis x label/.style={at={(ticklabel cs:0.5)},anchor=north},
	ytick={0,0.5,1},
	every axis y label/.style={at={(ticklabel cs:0.5)},rotate=90,anchor=south},
	scaled ticks=true,
	xlabel={$c$},
	]

\addplot[color=black, mark=none] plot coordinates {
(0.00796813,0.996058) (0.0159363,0.9922) (0.0239044,0.988425) (0.0318725,0.984732) (0.0398406,0.98112) (0.0478088,0.977588) (0.0557769,0.974136) (0.063745,0.970763) (0.0717131,0.967468) (0.0796813,0.964251) (0.0876494,0.96111) (0.0956175,0.958044) (0.103586,0.955054) (0.111554,0.952138) (0.119522,0.949296) (0.12749,0.946526) (0.135458,0.943829) (0.143426,0.941204) (0.151394,0.93865) (0.159363,0.936166) (0.167331,0.933752) (0.175299,0.931407) (0.183267,0.92913) (0.191235,0.926922) (0.199203,0.92478) (0.207171,0.922706) (0.215139,0.920698) (0.223108,0.918756) (0.231076,0.916878) (0.239044,0.915066) (0.247012,0.913317) (0.25498,0.911632) (0.262948,0.910011) (0.270916,0.908452) (0.278884,0.906955) (0.286853,0.90552) (0.294821,0.904146) (0.302789,0.902833) (0.310757,0.90158) (0.318725,0.900387) (0.326693,0.899253) (0.334661,0.898172) (0.342629,0.897142) (0.350598,0.896158) (0.358566,0.895218) (0.366534,0.894319) (0.374502,0.893458) (0.38247,0.892633) (0.390438,0.891842) (0.398406,0.891083) (0.406375,0.890353) (0.414343,0.889651) (0.422311,0.888976) (0.430279,0.888326) (0.438247,0.887699) (0.446215,0.887095) (0.454183,0.886512) (0.462151,0.885949) (0.47012,0.885405) (0.478088,0.88488) (0.486056,0.884371) (0.494024,0.883879) (0.501992,0.883403) (0.50996,0.882941) (0.517928,0.882494) (0.525896,0.88206) (0.533865,0.881639) (0.541833,0.881231) (0.549801,0.880834) (0.557769,0.880449) (0.565737,0.880074) (0.573705,0.87971) (0.581673,0.879356) (0.589641,0.879012) (0.59761,0.878676) (0.605578,0.87835) (0.613546,0.878032) (0.621514,0.877722) (0.629482,0.87742) (0.63745,0.877126) (0.645418,0.876838) (0.653386,0.876558) (0.661355,0.876285) (0.669323,0.876018) (0.677291,0.875757) (0.685259,0.875503) (0.693227,0.875254) (0.701195,0.875011) (0.709163,0.874773) (0.717131,0.874541) (0.7251,0.874314) (0.733068,0.874091) (0.741036,0.873874) (0.749004,0.873661) (0.756972,0.873453) (0.76494,0.873249) (0.772908,0.873049) (0.780876,0.872853) (0.788845,0.872661) (0.796813,0.872473) (0.804781,0.872289) (0.812749,0.872108) (0.820717,0.871931) (0.828685,0.871757) (0.836653,0.871587) (0.844622,0.87142) (0.85259,0.871256) (0.860558,0.871095) (0.868526,0.870937) (0.876494,0.870781) (0.884462,0.870629) (0.89243,0.870479) (0.900398,0.870332) (0.908367,0.870188) (0.916335,0.870046) (0.924303,0.869906) (0.932271,0.869769) (0.940239,0.869634) (0.948207,0.869502) (0.956175,0.869372) (0.964143,0.869243) (0.972112,0.869117) (0.98008,0.868993) (0.988048,0.868871) (0.996016,0.868751) (1.00398,0.868633) (1.01195,0.868517) (1.01992,0.868402) (1.02789,0.86829) (1.03586,0.868179) (1.04382,0.868069) (1.05179,0.867962) (1.05976,0.867856) (1.06773,0.867751) (1.0757,0.867649) (1.08367,0.867547) (1.09163,0.867447) (1.0996,0.867349) (1.10757,0.867252) (1.11554,0.867156) (1.12351,0.867062) (1.13147,0.866969) (1.13944,0.866877) (1.14741,0.866787) (1.15538,0.866698) (1.16335,0.86661) (1.17131,0.866523) (1.17928,0.866438) (1.18725,0.866353) (1.19522,0.86627) (1.20319,0.866188) (1.21116,0.866107) (1.21912,0.866027) (1.22709,0.865948) (1.23506,0.86587) (1.24303,0.865793) (1.251,0.865717) (1.25896,0.865642) (1.26693,0.865568) (1.2749,0.865495) (1.28287,0.865422) (1.29084,0.865351) (1.2988,0.865281) (1.30677,0.865211) (1.31474,0.865142) (1.32271,0.865074) (1.33068,0.865007) (1.33865,0.864941) (1.34661,0.864875) (1.35458,0.86481) (1.36255,0.864746) (1.37052,0.864683) (1.37849,0.864621) (1.38645,0.864559) (1.39442,0.864498) (1.40239,0.864437) (1.41036,0.864378) (1.41833,0.864319) (1.42629,0.86426) (1.43426,0.864202) (1.44223,0.864145) (1.4502,0.864089) (1.45817,0.864033) (1.46614,0.863978) (1.4741,0.863923) (1.48207,0.863869) (1.49004,0.863815) (1.49801,0.863762) (1.50598,0.86371) (1.51394,0.863658) (1.52191,0.863607) (1.52988,0.863556) (1.53785,0.863506) (1.54582,0.863456) (1.55378,0.863407) (1.56175,0.863358) (1.56972,0.86331) (1.57769,0.863262) (1.58566,0.863215) (1.59363,0.863168) (1.60159,0.863122) (1.60956,0.863076) (1.61753,0.863031) (1.6255,0.862986) (1.63347,0.862942) (1.64143,0.862897) (1.6494,0.862854) (1.65737,0.862811) (1.66534,0.862768) (1.67331,0.862725) (1.68127,0.862683) (1.68924,0.862642) (1.69721,0.862601) (1.70518,0.86256) (1.71315,0.862519) (1.72112,0.862479) (1.72908,0.86244) (1.73705,0.8624) (1.74502,0.862361) (1.75299,0.862323) (1.76096,0.862284) (1.76892,0.862246) (1.77689,0.862209) (1.78486,0.862172) (1.79283,0.862135) (1.8008,0.862098) (1.80876,0.862062) (1.81673,0.862026) (1.8247,0.86199) (1.83267,0.861955) (1.84064,0.86192) (1.84861,0.861885) (1.85657,0.861851) (1.86454,0.861817) (1.87251,0.861783) (1.88048,0.861749) (1.88845,0.861716) (1.89641,0.861683) (1.90438,0.86165) (1.91235,0.861618) (1.92032,0.861586) (1.92829,0.861554) (1.93625,0.861522) (1.94422,0.861491) (1.95219,0.861459) (1.96016,0.861429) (1.96813,0.861398) (1.9761,0.861368) (1.98406,0.861337) (1.99203,0.861308) 
};

\addplot[color=black, dashed, mark=none] plot coordinates {
(0.1,0.955665) (0.2,0.912307) (0.3,0.876424) (0.4,0.840397) (0.5,0.803483) (0.6,0.775519) (0.7,0.744882) (0.8,0.72158) (0.9,0.702408) (1.,0.674823) (1.1,0.648975) (1.2,0.629714) (1.3,0.616024) (1.4,0.587569) (1.5,0.575931) (1.6,0.561955) (1.7,0.54799) (1.8,0.530905) (1.9,0.517535) (2.,0.501236) 
};

\addplot[color=black, thick, mark=.] plot coordinates {
(0.00796813,0.996037) (0.0159363,0.992116) (0.0239044,0.988235) (0.0318725,0.984394) (0.0398406,0.980593) (0.0478088,0.976831) (0.0557769,0.973107) (0.063745,0.96942) (0.0717131,0.96577) (0.0796813,0.962157) (0.0876494,0.958579) (0.0956175,0.955036) (0.103586,0.951527) (0.111554,0.948052) (0.119522,0.944611) (0.12749,0.941203) (0.135458,0.937826) (0.143426,0.934482) (0.151394,0.931169) (0.159363,0.927886) (0.167331,0.924634) (0.175299,0.921412) (0.183267,0.918219) (0.191235,0.915055) (0.199203,0.91192) (0.207171,0.908812) (0.215139,0.905733) (0.223108,0.90268) (0.231076,0.899655) (0.239044,0.896656) (0.247012,0.893682) (0.25498,0.890735) (0.262948,0.887813) (0.270916,0.884916) (0.278884,0.882043) (0.286853,0.879195) (0.294821,0.876371) (0.302789,0.87357) (0.310757,0.870792) (0.318725,0.868038) (0.326693,0.865306) (0.334661,0.862596) (0.342629,0.859908) (0.350598,0.857242) (0.358566,0.854598) (0.366534,0.851975) (0.374502,0.849372) (0.38247,0.84679) (0.390438,0.844228) (0.398406,0.841687) (0.406375,0.839165) (0.414343,0.836662) (0.422311,0.834179) (0.430279,0.831715) (0.438247,0.82927) (0.446215,0.826843) (0.454183,0.824435) (0.462151,0.822044) (0.47012,0.819672) (0.478088,0.817317) (0.486056,0.814979) (0.494024,0.812659) (0.501992,0.810356) (0.50996,0.80807) (0.517928,0.8058) (0.525896,0.803546) (0.533865,0.801309) (0.541833,0.799088) (0.549801,0.796882) (0.557769,0.794692) (0.565737,0.792518) (0.573705,0.790359) (0.581673,0.788215) (0.589641,0.786085) (0.59761,0.783971) (0.605578,0.781871) (0.613546,0.779785) (0.621514,0.777714) (0.629482,0.775657) (0.63745,0.773614) (0.645418,0.771584) (0.653386,0.769568) (0.661355,0.767566) (0.669323,0.765577) (0.677291,0.763601) (0.685259,0.761638) (0.693227,0.759688) (0.701195,0.757751) (0.709163,0.755826) (0.717131,0.753914) (0.7251,0.752014) (0.733068,0.750126) (0.741036,0.74825) (0.749004,0.746387) (0.756972,0.744535) (0.76494,0.742694) (0.772908,0.740866) (0.780876,0.739049) (0.788845,0.737243) (0.796813,0.735448) (0.804781,0.733665) (0.812749,0.731892) (0.820717,0.73013) (0.828685,0.728379) (0.836653,0.726639) (0.844622,0.724909) (0.85259,0.72319) (0.860558,0.721481) (0.868526,0.719783) (0.876494,0.718094) (0.884462,0.716416) (0.89243,0.714747) (0.900398,0.713088) (0.908367,0.711439) (0.916335,0.7098) (0.924303,0.70817) (0.932271,0.70655) (0.940239,0.704939) (0.948207,0.703337) (0.956175,0.701745) (0.964143,0.700162) (0.972112,0.698587) (0.98008,0.697022) (0.988048,0.695465) (0.996016,0.693918) (1.00398,0.692379) (1.01195,0.690848) (1.01992,0.689326) (1.02789,0.687813) (1.03586,0.686308) (1.04382,0.684811) (1.05179,0.683323) (1.05976,0.681842) (1.06773,0.68037) (1.0757,0.678906) (1.08367,0.677449) (1.09163,0.676001) (1.0996,0.67456) (1.10757,0.673127) (1.11554,0.671702) (1.12351,0.670284) (1.13147,0.668874) (1.13944,0.667471) (1.14741,0.666076) (1.15538,0.664688) (1.16335,0.663307) (1.17131,0.661934) (1.17928,0.660567) (1.18725,0.659208) (1.19522,0.657856) (1.20319,0.65651) (1.21116,0.655172) (1.21912,0.65384) (1.22709,0.652516) (1.23506,0.651198) (1.24303,0.649886) (1.251,0.648581) (1.25896,0.647283) (1.26693,0.645992) (1.2749,0.644706) (1.28287,0.643427) (1.29084,0.642155) (1.2988,0.640889) (1.30677,0.639629) (1.31474,0.638375) (1.32271,0.637127) (1.33068,0.635886) (1.33865,0.63465) (1.34661,0.633421) (1.35458,0.632197) (1.36255,0.63098) (1.37052,0.629768) (1.37849,0.628562) (1.38645,0.627362) (1.39442,0.626168) (1.40239,0.624979) (1.41036,0.623796) (1.41833,0.622618) (1.42629,0.621446) (1.43426,0.62028) (1.44223,0.619119) (1.4502,0.617963) (1.45817,0.616813) (1.46614,0.615668) (1.4741,0.614528) (1.48207,0.613394) (1.49004,0.612265) (1.49801,0.611141) (1.50598,0.610022) (1.51394,0.608908) (1.52191,0.607799) (1.52988,0.606696) (1.53785,0.605597) (1.54582,0.604503) (1.55378,0.603415) (1.56175,0.602331) (1.56972,0.601252) (1.57769,0.600177) (1.58566,0.599108) (1.59363,0.598043) (1.60159,0.596983) (1.60956,0.595928) (1.61753,0.594877) (1.6255,0.593831) (1.63347,0.592789) (1.64143,0.591752) (1.6494,0.590719) (1.65737,0.589691) (1.66534,0.588668) (1.67331,0.587648) (1.68127,0.586634) (1.68924,0.585623) (1.69721,0.584617) (1.70518,0.583615) (1.71315,0.582617) (1.72112,0.581624) (1.72908,0.580635) (1.73705,0.57965) (1.74502,0.578669) (1.75299,0.577692) (1.76096,0.576719) (1.76892,0.57575) (1.77689,0.574786) (1.78486,0.573825) (1.79283,0.572868) (1.8008,0.571916) (1.80876,0.570967) (1.81673,0.570022) (1.8247,0.569081) (1.83267,0.568144) (1.84064,0.56721) (1.84861,0.566281) (1.85657,0.565355) (1.86454,0.564433) (1.87251,0.563514) (1.88048,0.5626) (1.88845,0.561689) (1.89641,0.560781) (1.90438,0.559877) (1.91235,0.558977) (1.92032,0.558081) (1.92829,0.557188) (1.93625,0.556298) (1.94422,0.555412) (1.95219,0.55453) (1.96016,0.55365) (1.96813,0.552775) (1.9761,0.551903) (1.98406,0.551034) (1.99203,0.550168) 
};

\addplot[color=black, thick, mark=.] plot coordinates {
(0.1,0.95105) (0.2,0.9082) (0.3,0.87045) (0.4,0.8336) (0.5,0.79725) (0.6,0.76855) (0.7,0.73775) (0.8,0.71585) (0.9,0.69465) (1.,0.66735) (1.1,0.6439) (1.2,0.6234) (1.3,0.61035) (1.4,0.58125) (1.5,0.57175) (1.6,0.55575) (1.7,0.543) (1.8,0.52455) (1.9,0.51175) (2.,0.49635) 
};

\addplot[color=black, dashed, mark=none] plot coordinates {
(0.1,0.946435) (0.2,0.904093) (0.3,0.864476) (0.4,0.826803) (0.5,0.791017) (0.6,0.761581) (0.7,0.730618) (0.8,0.71012) (0.9,0.686892) (1.,0.659877) (1.1,0.638825) (1.2,0.617086) (1.3,0.604676) (1.4,0.574931) (1.5,0.567569) (1.6,0.549545) (1.7,0.53801) (1.8,0.518195) (1.9,0.505965) (2.,0.491464) 
};

\addplot[color=black, mark=none] plot coordinates {
(0.00796813,0.996026) (0.0159363,0.992074) (0.0239044,0.988142) (0.0318725,0.984232) (0.0398406,0.980342) (0.0478088,0.976472) (0.0557769,0.972623) (0.063745,0.968794) (0.0717131,0.964985) (0.0796813,0.961197) (0.0876494,0.957428) (0.0956175,0.953679) (0.103586,0.94995) (0.111554,0.946241) (0.119522,0.94255) (0.12749,0.93888) (0.135458,0.935228) (0.143426,0.931596) (0.151394,0.927983) (0.159363,0.924388) (0.167331,0.920812) (0.175299,0.917255) (0.183267,0.913717) (0.191235,0.910197) (0.199203,0.906695) (0.207171,0.903212) (0.215139,0.899747) (0.223108,0.8963) (0.231076,0.89287) (0.239044,0.889459) (0.247012,0.886065) (0.25498,0.882689) (0.262948,0.87933) (0.270916,0.875989) (0.278884,0.872665) (0.286853,0.869358) (0.294821,0.866068) (0.302789,0.862796) (0.310757,0.85954) (0.318725,0.856301) (0.326693,0.853079) (0.334661,0.849873) (0.342629,0.846684) (0.350598,0.843511) (0.358566,0.840354) (0.366534,0.837214) (0.374502,0.83409) (0.38247,0.830982) (0.390438,0.82789) (0.398406,0.824813) (0.406375,0.821753) (0.414343,0.818708) (0.422311,0.815678) (0.430279,0.812664) (0.438247,0.809666) (0.446215,0.806682) (0.454183,0.803714) (0.462151,0.800761) (0.47012,0.797824) (0.478088,0.794901) (0.486056,0.791992) (0.494024,0.789099) (0.501992,0.78622) (0.50996,0.783356) (0.517928,0.780507) (0.525896,0.777672) (0.533865,0.774851) (0.541833,0.772044) (0.549801,0.769252) (0.557769,0.766474) (0.565737,0.763709) (0.573705,0.760959) (0.581673,0.758222) (0.589641,0.7555) (0.59761,0.752791) (0.605578,0.750095) (0.613546,0.747413) (0.621514,0.744745) (0.629482,0.74209) (0.63745,0.739448) (0.645418,0.73682) (0.653386,0.734204) (0.661355,0.731602) (0.669323,0.729013) (0.677291,0.726436) (0.685259,0.723873) (0.693227,0.721322) (0.701195,0.718784) (0.709163,0.716259) (0.717131,0.713746) (0.7251,0.711245) (0.733068,0.708757) (0.741036,0.706282) (0.749004,0.703818) (0.756972,0.701367) (0.76494,0.698928) (0.772908,0.696501) (0.780876,0.694086) (0.788845,0.691683) (0.796813,0.689292) (0.804781,0.686913) (0.812749,0.684545) (0.820717,0.682189) (0.828685,0.679844) (0.836653,0.677511) (0.844622,0.67519) (0.85259,0.67288) (0.860558,0.670581) (0.868526,0.668294) (0.876494,0.666017) (0.884462,0.663752) (0.89243,0.661498) (0.900398,0.659255) (0.908367,0.657023) (0.916335,0.654802) (0.924303,0.652591) (0.932271,0.650392) (0.940239,0.648203) (0.948207,0.646024) (0.956175,0.643857) (0.964143,0.641699) (0.972112,0.639553) (0.98008,0.637416) (0.988048,0.63529) (0.996016,0.633175) (1.00398,0.631069) (1.01195,0.628974) (1.01992,0.626889) (1.02789,0.624813) (1.03586,0.622748) (1.04382,0.620693) (1.05179,0.618648) (1.05976,0.616612) (1.06773,0.614586) (1.0757,0.61257) (1.08367,0.610564) (1.09163,0.608567) (1.0996,0.60658) (1.10757,0.604602) (1.11554,0.602634) (1.12351,0.600675) (1.13147,0.598726) (1.13944,0.596785) (1.14741,0.594854) (1.15538,0.592933) (1.16335,0.59102) (1.17131,0.589117) (1.17928,0.587222) (1.18725,0.585336) (1.19522,0.58346) (1.20319,0.581592) (1.21116,0.579733) (1.21912,0.577883) (1.22709,0.576042) (1.23506,0.574209) (1.24303,0.572385) (1.251,0.57057) (1.25896,0.568763) (1.26693,0.566964) (1.2749,0.565175) (1.28287,0.563393) (1.29084,0.56162) (1.2988,0.559855) (1.30677,0.558098) (1.31474,0.55635) (1.32271,0.55461) (1.33068,0.552878) (1.33865,0.551154) (1.34661,0.549438) (1.35458,0.54773) (1.36255,0.54603) (1.37052,0.544338) (1.37849,0.542653) (1.38645,0.540977) (1.39442,0.539308) (1.40239,0.537648) (1.41036,0.535994) (1.41833,0.534349) (1.42629,0.532711) (1.43426,0.53108) (1.44223,0.529458) (1.4502,0.527842) (1.45817,0.526234) (1.46614,0.524634) (1.4741,0.523041) (1.48207,0.521455) (1.49004,0.519876) (1.49801,0.518305) (1.50598,0.516741) (1.51394,0.515184) (1.52191,0.513634) (1.52988,0.512091) (1.53785,0.510556) (1.54582,0.509027) (1.55378,0.507505) (1.56175,0.50599) (1.56972,0.504482) (1.57769,0.502981) (1.58566,0.501487) (1.59363,0.5) (1.60159,0.498519) (1.60956,0.497045) (1.61753,0.495578) (1.6255,0.494117) (1.63347,0.492663) (1.64143,0.491216) (1.6494,0.489775) (1.65737,0.48834) (1.66534,0.486912) (1.67331,0.48549) (1.68127,0.484075) (1.68924,0.482666) (1.69721,0.481264) (1.70518,0.479868) (1.71315,0.478478) (1.72112,0.477094) (1.72908,0.475716) (1.73705,0.474345) (1.74502,0.472979) (1.75299,0.47162) (1.76096,0.470267) (1.76892,0.46892) (1.77689,0.467579) (1.78486,0.466243) (1.79283,0.464914) (1.8008,0.463591) (1.80876,0.462273) (1.81673,0.460962) (1.8247,0.459656) (1.83267,0.458355) (1.84064,0.457061) (1.84861,0.455772) (1.85657,0.454489) (1.86454,0.453212) (1.87251,0.45194) (1.88048,0.450674) (1.88845,0.449414) (1.89641,0.448158) (1.90438,0.446909) (1.91235,0.445665) (1.92032,0.444426) (1.92829,0.443193) (1.93625,0.441965) (1.94422,0.440742) (1.95219,0.439525) (1.96016,0.438313) (1.96813,0.437107) (1.9761,0.435905) (1.98406,0.434709) (1.99203,0.433518) 
};

\draw (1.75,0.81) node [fill=none, draw=none, text=black, opacity=1, text opacity=1] {$T^{\textrm{upper}}$};
\draw (1.8,0.63) node [fill=none, draw=none, text=black, opacity=1, text opacity=1] {$\frac{\ln{(1+c)}}{c}$};
\draw (1.75,0.41) node [fill=none, draw=none, text=black, opacity=1, text opacity=1] {$T^{\textrm{lower}}$};

\end{axis}
\end{tikzpicture}